\documentclass[11pt,reqno]{amsart}
\newcommand{\mysection}[1]{\section{#1}
      \setcounter{equation}{0}}

\newcommand\cbrk{\text{$]$\kern-.15em$]$}} 
\newcommand\opar{\text{\raise.2ex\hbox{${\scriptstyle | }$}\kern-.34em$($} }

\usepackage{color}
\usepackage{amsmath,amsthm,amssymb,amsfonts,enumerate,color,enumerate}
\usepackage[pdftex]{graphicx}

\usepackage{verbatim}
\usepackage[utf8]{inputenc}
\usepackage[spanish,USenglish]{babel} 
\usepackage{color}
\usepackage{tikz}
\usepackage[hidelinks]{hyperref}
\usepackage{mathtools}
\usepackage{bbm}
\usepackage{mathabx}
\allowdisplaybreaks

\newtheorem{theorem}{Theorem}[section]
\newtheorem{lemma}[theorem]{Lemma}
\newtheorem{proposition}[theorem]{Proposition}
\newtheorem{corollary}[theorem]{Corollary}

\theoremstyle{definition}
\newtheorem{assumption}{Assumption}[section]
\newtheorem{definition}{Definition}[section]

\theoremstyle{remark}
\newtheorem{remark}{Remark}[section]

\newcommand{\F}{\mathcal{{F}}}

\newcommand\bR{\mathbb{R}}
\newcommand\bP{\mathbb{P}}

\newcommand\bM{\mathbb{M}}

\newcommand\bQ{\mathbb{Q}}

\newcommand\frZ{\mathfrak{Z}}
\newcommand\frz{\mathfrak{z}}

\newcommand\cB{\mathcal{B}}

\newcommand\cF{\mathcal{F}}
\newcommand\cG{\mathcal{G}}
\newcommand\cH{\mathcal{H}}

\newcommand\cL{\mathcal{L}}
\newcommand\cM{\mathcal{M}}

\newcommand\cP{\mathcal{P}}
\newcommand\cO{\mathcal{O}}

\newcommand\cY{\mathcal{Y}}
\newcommand\cZ{\mathcal{Z}}

\newcommand{\E}{\mathbb{{E}}}
\newcommand{\1}{\mathbbm{1}}

\newcommand{\Ntn}{\tilde{N}_{0}}
\newcommand{\Nte}{\tilde{N}_{1}}

\newcommand{\vp}{\varphi}

\newcommand{\bRYn}{{\mathbb{R}^{d'}\setminus\{0\}}}

\makeatletter
 \newcommand{\sumstar}
 {\operatornamewithlimits{\sum@\kern-.2em\raise1ex\hbox{*}}}
 \makeatother

\begin{document}

\author[F. Germ]{Fabian Germ}
\address{School of Mathematics and Maxwell Institute, University of Edinburgh, Scotland, United Kingdom}
\email{F.Germ@ed.ac.uk}

\author[I. Gy\"ongy]{Istv\'an Gy\"ongy}
\address{School of Mathematics and Maxwell Institute, 
University of Edinburgh, Scotland, United Kingdom.}
\email{i.gyongy@ed.ac.uk}



\begin{abstract} 
This paper is the first part of a series of papers on filtering for partially observed 
jump diffusions satisfying a stochastic differential equation driven by Wiener processes 
and  Poisson martingale measures. The coefficients of the equation only
satisfy appropriate growth conditions. 
Some results in filtering theory of diffusion processes are extended to jump diffusions  
and equations for the time evolution of the conditional distribution 
and the unnormalised conditional distribution of the unobserved process at time $t$, 
given the observations until $t$, are presented. 
\end{abstract}

\title[The Filtering equations]{On partially observed jump diffusions I. The Filtering equations
}

\maketitle


\mysection{Introduction}

This is the first part of a series of papers on filtering of jump-diffusions. 
We consider on a given complete 
filtered probability space $(\Omega,\F,(\F_t)_{t\geq 0},P)$ 
a $d+d'$-dimensional stochastic process $(Z_t)_{t\in[0,T]}=(X_t,Y_t)_{t\in[0,T]}$, 
satisfying the stochastic differential equation
\begin{equation}                                                                          
\begin{split}
    dX_t    &= b(t,Z_t)dt + \sigma(t,Z_t)dW_t + \rho(t,Z_t)dV_t\\
            &+\int_{\frZ_0}\eta(t, Z_{t-},\frz)\,\tilde N_0(d\frz,dt) 
            + \int_{\frZ_1}\xi(t,Z_{t-},\frz)\,\tilde N_{1}(d\frz,dt)\\
    dY_t    &=B(t,Z_t)dt + dV_t + \int_{\frZ_1} \frz\,\tilde N_{1}(d\frz, dt),
    \end{split}                                                                                             \label{system_1}
\end{equation} 
on the interval $[0,T]$ for a given $\cF_0$-measurable initial value $Z_0=(X_0,Y_0)$, 
where 
$(W_t,V_t)_{t\geq0}$ is $d_1+d'$-dimensional $\cF_t$-Wiener process, 
and $\tilde N_i(d\frz,dt) = N_i(d\frz,dt)-\nu_i(d\frz)dt$, $i=0,1$, are independent 
$\F_t$-Poisson martingale measures 
on $\bR_{+}\times\frZ_i$ 
with $\sigma$-finite characteristic measures $\nu_0$ and $\nu_1$ 
on separable measurable spaces 
$(\frZ_0,\cZ_0)$ and $(\frZ_1,\cZ_1)=(\bRYn,\cB(\bRYn))$, respectively. 
The mappings $b=(b^i)$, $B=(B^i)$, $\sigma=(\sigma^{ij})$ and $\rho=(\rho^{il})$ 
are Borel functions of $(t,z)=(t,x,y)\in\bR_+\times\bR^{d+d'}$, 
with values in $\bR^d$, $\bR^{d'}$, $\bR^{d\times d_1}$ 
and $\bR^{d\times d'}$, respectively, and $\eta=(\eta^i)$ and $\xi=(\xi^i)$ are 
$\bR^d$-valued $\cB(\bR_+\times\bR^{d+d'})\otimes\cZ_0$-measurable and 
$\bR^d$-valued $\cB(\bR_+\times\bR^{d+d'})\otimes\cZ_1$-measurable functions of  
$(t,z,\frz_0)\in\bR_+\times\bR^{d+d'}\times \frZ_0$ 
and $(t,z,\frz_1)\in\bR_+\times\bR^{d+d'}\times \frZ_1$, 
respectively.  

We are concerned with the classic task of filtering theory: to calculate   
at each time $t$ the mean square estimate of $f(X_t)$, a real-valued Borel function  
of the ``unobservable component" $X_t$ of $Z_t$, given 
the ``observations" $\{Y_s:s\leq t\}$. Since, as it is well-known,  this estimate is 
the conditional expectation 
$$
\E (f(X_t)|Y_s:s\leq t)=\int_{\bR^d}f(x)P_t(dx), 
\quad t\in[0,T], 
$$
we are interested in equations for the evolution of $P_t(dx)$, 
the conditional distribution of $X_t$ given $\{Y_s,s\leq t\}$. 
Their derivation for a large class of coefficients is the aim of this paper. 
In the subsequent papers of this series we  investigate the existence of the conditional density $\pi_t(x)=P_t(dx)/dx$ and its regularity 
properties.  

There has been an immense interest in the development of filtering theory 
due to its wide applicability in various disciplines, be they of theoretical or applied nature. 
A vast amount of research has been done on filtering of partially observed processes 
governed by stochastic differential equations driven by Wiener processes, i.e., when 
$\eta=\xi=0$ in \eqref{system_1}, 
and a quite complete nonlinear filtering theory was built up, see for instance \cite{DC2014} 
for a historical account.
\newline
In this case it is well-known that $(P_t(dx))_{t\geq0}$ 
satisfies a nonlinear stochastic PDE (SPDE), often called the Kushner-Shiryayev 
equation in filtering theory. It is also well-known that this equation 
can be transformed into a linear SPDE, called Zakai equation, or Duncan-Mortensen-Zakai 
equation for $\mu_t(dx)=\lambda_tP_t(dx)$, the unnormalised conditional distribution, 
where $(\lambda_t)_{t \in[0,T] }$ is a positive normalizing stochastic process. 

There exist several known methods of deriving the filtering equations for partially 
observed diffusion processes,  three prominent of which are the ``innovation method", 
the ``reference measure method" and a ``direct approach". The innovation method 
 is based on ``innovation process" representations, (see \cite{LS1974} 
 and \cite{FKK1972}), and the 
direct approach is based on suitable existence and uniqueness theorems 
for stochastic PDEs (see \cite{KZ2000}). 
The reference probability method is employed in this paper, where we make use of the 
fact that by Girsanov's theorem one can introduce a new measure under which 
the observation $\sigma$-algebra, $\sigma(Y_s:s\leq t)$, is the product $\sigma$-algebra 
of three independent $\sigma$-algebras: the $\sigma$-algebra generated by 
the initial observation $Y_0$ and the $\sigma$-algebras generated by  
 the Wiener process and the Poisson random measure in the observation process 
until time $t$, respectively. This structure of the observation $\sigma$-algebra makes it possible to 
calculate conditional expectations of functions of the process $Z$ given the observations. 
(See, e.g., \cite{CCC2014} for descriptions of various methods used in filtering theory.)

Recently, also filtering for jump-diffusion systems have been intensively studied, 
which are most often modelled as SDEs driven by Wiener processes 
as well as random jump measures, 
a classical case of which are Poisson random measures.
 In an early article thereon, \cite{P2005}, the filtering equations 
 were derived for uncorrelated continuous observations, as well 
 as an observation process driven only by a jump process that has 
 no common jumps with the signal. A similar system with continuous 
 uncorrelated observations has also been considered in \cite{SS2009}. 
 A more general nonlinear system with jumps in the observation process was considered in \cite{B2014}. In \cite{A} the filtering equations for a large class of uncorrelated linear systems with jumps are derived. 
In \cite{GM2011} a very general model is considered and a representation 
for optional projection of the signal process 
is derived. However, due to the generality 
a number of additional assumptions  are imposed on their model
and equations for the filtering measures are not obtained. 
\newline
In \cite{CC2012} and \cite{CC2014} 
the authors deal with a one-dimensional jump-diffusion where observation 
and signal may have common jumps by introducing a new random measure, 
nonzero only for observable jumps, relying on a construction in \cite{C2006}. 
However, they impose a finiteness condition on the support of the integrand in front of the jump term, which translates to observing only finitely many jumps almost surely. 
In such a case, the jump\ measure and the associated compensator, also referred to as dual predictable projection, allow for a specific decomposition, see for instance \cite[Sec. XI.4]{HWY}. The filtering equations have been derived for a class 
of jump diffusion systems \cite{QD2015}, later generalised 
to include correlated Wiener process noises in \cite{Q2019}, 
however it seems to us that 
certain important results needed for this derivation, 
including Lemma 3.2 in \cite{QD2015}, also used in \cite{Q2019}, 
do not hold, for instance if one considers the case of vanishing coefficients. 
A model where a correlation structure between 
the L\'evy process noises in signal and observation is described using copulas is used in \cite{FH2018} to derive 
the Zakai equation.\newline

In this paper we obtain  the filtering equations 
for a jump-diffusion system driven by correlated Wiener process, 
as well as correlated Poisson martingale measure noises. 
We impose common  linear growth conditions. 
We do not assume any non-degeneracy conditions  
and allow for the number of jumps in any component  
of $(Z_t)_{t\geq0}$ to be infinite over finite intervals. 
In order to obtain the equations, we generalise some 
results from filtering theory and in particular prove a 
projection theorem for a wide class of functions. \newline
In Section \ref{sec main results} a fairly general 
condition for Girsanov's transformation and our main result are presented. 
In Section \ref{section preliminaries} a projection theorem covering 
a wide class of processes is proven,  and thereby in the last section 
the filtering equations are derived. 
\newline
Conditions and results on the existence and 
regularity of the filtering density are presented 
in subsequent articles of this series.

We conclude with some notions and notations 
used throughout the paper. 
For an integer $n\geq0$ the notation $C^n_b(\bR^d)$ means 
the space of real-valued bounded continuous functions on $\bR^d$, 
which have bounded and continuous derivatives up to order $n$. 
(If $n=0$, then $C^0_b(\bR^d)=C_b(\bR^d)$ denotes the space of 
real-valued bounded continuous functions on $\bR^d$). 
We denote by  $\bM=\bM(\bR^d)$ the set of finite Borel measures 
on $\bR^d$. For $\mu\in\bM$ we use the notation 
$$
\mu(\varphi)=\int_{\bR^d}\varphi(x)\,\mu(dx) 
$$
for Borel functions $\varphi$ on $\bR^d$. 
We say that a function $\nu:\Omega\to\bM$ is $\cG$-measurable 
for a $\sigma$-algebra $\cG\subset\cF$, if $\nu(\varphi)$ is a  
$\cG$-measurable random variable for every bounded Borel function 
$\varphi$ on $\bR^d$. 
An $\bM$-valued stochastic process $\nu=(\nu_t)_{t\in[0,T]}$ 
is said to be weakly cadlag if almost surely 
$\nu_t(\varphi)$ is a cadlag function of $t$ for all $\varphi\in C_b(\bR^d)$. 
For such a process $\nu$ there is a set 
$\Omega'\subset\Omega$ of full probability and there is uniquely defined 
(up to indistinguishability)   $\bM$-valued processes 
$(\nu_{t-})_{t\in(0,T]}$ 
such that for every $\omega\in\Omega'$ 
$$
\nu_{t-}(\varphi)=\lim_{s\uparrow t}\nu_{s}(\varphi)
\quad\text{for all $\varphi\in C_b(\bR^d)$
and $t\in(0,T]$,}
$$
and for each $\omega\in\Omega'$ we have $\nu_{t-}=\nu_t$, 
for all but at most countably many $t\in(0,T]$. 
For processes $U=(U_t)_{t\in [0,T]}$ we use the notation
$
\cF_t^{U}
$
for the $P$-completion of the $\sigma$-algebra generated by 
$\{U_s: s\leq t\}$.  For a  measure space 
$(\frZ,\cZ,\nu)$ and $p\geq1$ we use the notation 
$L_p(\frZ)$ for the $L_p$-space of $\bR^d$-valued 
$\cZ$-measurable processes defined on $\frZ$.  
For $\sigma$-algebras $\cG_i\subset\cF$, $i=1,2$,  the notation 
$\cG_1\vee\cG_2$ means the $P$-completion of the smallest 
$\sigma$-algebra containing $G_i$ for $i=1,2$.
Finally, we always use without mention the summation convention, 
by which repeated integer valued indices imply a summation.

\mysection{Formulation of the main results}
\label{sec main results}

We consider on a given finite interval $[0,T]$ a $d+d'$-dimensional 
stochastic process  $Z=(Z_t)_{t\in[0,T]}=(X_t,Y_t)_{t\in[0,T]}$  
carried by a complete probability space $(\Omega,\cF,P)$, 
equipped with a filtration $(\cF_t)_{t\geq0}$ such that $\cF_0$ 
contains the $P$-null sets of $\cF$. 
We assume that $Z$ satisfies the stochastic differential equation \eqref{system_1} 
on the interval $[0,T]$, with an $\cF_0$-measurable initial value $Z_0=(X_0,Y_0)$. 

Besides the natural measurability conditions  on 
the coefficients $b$, $\sigma$, $\rho$, $\xi$, $\eta$ and $B$, described in the Introduction, 
we assume the following conditions. 

\begin{assumption}                                                    \label{assumption SDE}
\begin{enumerate}
\item[(i)]
There are nonnegative constants $K_0$, $K_1$ and $K_2$ 
such that 
$$
|b(t,z)|^2 \leq K_0+K_1|z|^2, 
\quad 
|\sigma(t,z)|^2+ |\rho(t,z)|^2+|B(t,z)|^2\leq K_0+K_2|z|^2,
$$
$$
|\eta(t,z)|^2_{L_2(\frZ_0)}+  
|\xi(t,z)|^{2}_{L_2(\frZ_1)}
\leq K_0+K_2|z|^2, \quad 
\int_{\frZ_1}|\frz|^2\nu_1(d\frz)\leq K_0 
$$
for all $z=(x,y)\in\mathbb{R}^{d+d'}$ and  $t\in[0,T]$,  
and we have 
\item[(ii)] 
\begin{equation}                                                                \label{moments}
K_1\E|X_0|+K_2\E|X_0|^2<\infty. 
\end{equation}
\end{enumerate}
\end{assumption}

Note that in \eqref{moments} we use 
the convention that $0\times\infty=0$, i.e., 
if $K_2=0$, then the finiteness of the second moment of 
$|X_0|$ is not required, and if $K_1=K_2=0$ then 
Assumption \ref{assumption SDE} (ii) clearly holds.

The following moment estimate is known and can be easily proved by the 
help of well-known martingale inequalities. 
\begin{remark} If Assumption \ref{assumption SDE}(i) holds, then 
for every $p\in[1,2]$ and $A\in\cF_0$ we have 
\begin{equation}                                                              \label{bound_Z}
\E\sup_{t\leq T}{\bf1}_A|Z_t|^p\leq N(1+\E{\bf1}_A|Z_0|^p)
\end{equation}
with a constant $N$ depending only on 
$p$, $T$, $K_0$, $K_1$, $K_2$ and $d+d'$.
\end{remark}

We make also the following assumption. 

\begin{assumption}                                                                              \label{assumption Girsanov}
We have $\E\gamma_T=1$, where 
\begin{equation}                                                                                   \label{gamma}
\gamma_t
=\exp\left(
-\int_0^tB(s,X_s,Y_s)\,dV_s-\tfrac{1}{2}\int_0^t|B(s,X_s,Y_s)|^2\,ds
\right), 
\quad t\in[0,T].
\end{equation}
\end{assumption}                                                                                 
This assumption implies that the measure $Q$, 
defined by $dQ=\gamma_TdP$ 
on $\cF$, is a probability measure, 
and hence by Girsanov's theorem under $Q$
the process 
\begin{equation}                                                                                      \label{tilde Wiener}
\tilde V_t=\int_0^tB(s,X_s,Y_s)\,ds+V_t,\quad t\in[0,T], 
\end{equation}
is an $\cF_t$-Wiener process. 

To describe the evolution of the conditional distribution 
$P_t(dx)=P(X_t\in dx|Y_s,s\leq t)$ 
for $t\in[0,T]$, we introduce 
the random differential operators 
$$
\cL_t=a^{ij}_t(x)D_{ij}+b^i_t(x)D_i, 
\quad 
\cM^k_t=\rho_t^{ik}(x)D_i+B^k_t(x), \quad k=1,2,...,d', 
$$
where 
$$
a^{ij}_t(x):=\tfrac{1}{2}\sum_{k=1}^{d_1}(\sigma^{ik}_t\sigma^{jk}_t)(x)
+\tfrac{1}{2}\sum_{l=1}^{d'}(\rho^{il}_t\rho_t^{jl})(x), 
\quad\sigma_t^{ik}(x):=\sigma^{ik}(t,x,Y_t),\quad
\rho_t^{il}(x):=\rho^{il}(t,x,Y_t), 
$$
$$
b^i_t(x):=b^i(t,x,Y_t),
\quad 
B^k_t(x):=
B^k(t,x,Y_t)
$$
for $\omega\in\Omega$, $t\in[0,T]$, $x=(x^1,...,x^d)\in\bR^d$, 
and $D_i=\partial/\partial x^i$, 
$D_{ij}=\partial^2/(\partial x^i\partial x^j)$ for $i,j=1,2...,d$. 
Moreover for every $t\in[0,T]$ and $\frz \in \frZ_1$ 
we introduce the random operators $I_t^{\xi}$ and $J_t^{\xi}$ defined by 
\begin{equation}                                                                      \label{IJ}                                                       
I_t^{\xi}\varphi(x,\frz)=\varphi(x+\xi_t(x,\frz), \frz)-\varphi(x,\frz), 
\quad
J_t^{\xi}\phi(x, \frz)=I_t^{\xi}\phi(x, \frz)-\sum_{i=1}^d\xi_t^i(x,\frz)D_i\phi(x,\frz)
\end{equation}
for functions $\varphi=\varphi(x,\frz)$ and $\phi=\phi(x,\frz)$ of 
$x\in\bR^d$ and $\frz\in\frZ_1$, 
and furthermore the random operators 
$I_t^{\eta}$ and $J_t^{\eta}$, defined as $I_t^{\xi}$ and $J_t^{\xi}$, respectively, with 
$\eta_t(x,\frz)$ in place of $\xi_t(x,\frz)$, where 
$$
\xi_t(x,\frz_{1}):=\xi(t,x,Y_{t-},\frz_{1}),
\quad
\eta_t(x,\frz_{0}):=\eta(t,x,Y_{t-},\frz_{0})
$$
for $\omega\in\Omega$, $t\in[0,T]$, 
$x\in\bR^d$ and $\frz_i\in\frZ_i$ for $i=0,1$ 
($Y_{0-}:=Y_0$). 

Now we are in the position to formulate our main result. 
Recall that we denote by $(\cF^Y_t)_{t\in[0,T]}$ the completed 
filtration generated by $(Y_t)_{t\in[0,T]}$. 
\begin{theorem}                                                                                                  \label{theorem Z1}
Let Assumptions \ref{assumption SDE} and  \ref{assumption Girsanov} hold. 
Then there exist measure-valued $\cF^Y_t$-adapted weakly cadlag processes 
$(P_t)_{t\in[0,T]}$ and $(\mu_t)_{t\in[0,T]}$ 
such that 
$$
P_t(\varphi)=\mu_t(\varphi)/\mu_t({\bf 1}),\quad \text{for $\omega\in\Omega,\,\, t\in[0,T]$}, 
$$
$$
P_t(\varphi)=\E(\varphi(X_t)|\cF^Y_t),\quad \mu_t(\varphi)=\E_{Q}(\gamma_t^{-1}\varphi(X_t)|\cF^Y_t) 
\quad\text{(a.s.) for each $t\in[0,T]$}, 
$$
for bounded Borel functions $\varphi$ on $\bR^d$, 
and for every $\varphi\in C^{2}_b(\bR^d)$ almost surely \begin{equation}
\begin{split}
\mu_t(\vp)=&  \mu_0(\vp) +  \int_0^t\mu_{s}(\cL_s\varphi)\,ds
+ \int_0^t \mu_{s}(\cM_s^k\varphi)\,d\tilde V^k_s
+ \int_0^t\int_{\frZ_0}\mu_{s}(J_s^{\eta}\varphi)\,\nu_0(d\frz)ds\\ 
&+ \int_0^t\int_{\frZ_1}\mu_{s}(J_s^{\xi}\varphi)\,\nu_1(d\frz)ds
+\int_0^t\int_{\frZ_1}\mu_{s-}(I_s^{\xi}\varphi)\,\tilde N_1(d\frz,ds), 
\end{split}
                                                                                                                                     \label{eqZ1}
\end{equation}
and
\begin{equation}
\begin{split}
P_t(\vp)=&  P_0(\vp) +  \int_0^tP_{s}(\cL_s\varphi)\,ds
+ \int_0^t \big(P_{s}(\cM_s^k\varphi)-P_{s}(\varphi)P_s(B^k_s)\big)\,d\bar V^k_s\\ 
&+ \int_0^t\int_{\frZ_0}P_{s}(J_s^{\eta}\varphi)\,\nu_0(d\frz)ds
+ \int_0^t\int_{\frZ_1}P_{s}(J_s^{\xi}\varphi)\,\nu_1(d\frz)ds\\
&+\int_0^t\int_{\frZ_1}P_{s-}(I_s^{\xi}\varphi)\,\tilde N_1(d\frz,ds)\\
\end{split}
                                                                                                                                      \label{eqZ2}
\end{equation}
for all $t\in[0,T]$, 
where $(\tilde V_t)_{t\in[0,T]}$ is given in \eqref{tilde Wiener}, 
and the process $(\bar V_t)_{t\in[0,T]}$ is defined by 
$$
d\bar V_t=d\tilde V_t-P_t(B_t)\,dt=dV_t+(B_t(X_t)-P_t(B_t))\,dt, \quad \bar V_0=0. 
$$
\end{theorem}
\begin{remark}
Clearly, $\bar V=(\bar V_t)_{t\in[0,T]}$ is a continuous process, 
starting from zero, 
and by the help of Lemma \ref{lemma sa} below it is easy to see that 
it is  $\cF^Y_t$-adapted.
Moreover, it is not difficult to see that $\bar V$ is a martingale 
(under $P$) 
with respect to $(\cF_t)_{t\geq0}$, 
with quadratic variation process $[\bar V]_t=t$, $t\in[0,T]$. 
Hence by L\'evy's theorem, $\bar V$ is an $\cF^Y_t$-Wiener 
process.  It is called the {\it innovation process} in the case 
when the observation process 
does not have a stochastic integral component with respect to Poisson measures, 
i.e., when $\nu_1=0$. In this case  
it was conjectured that  $(\bar V_s)_{s\in[0,t]}$ together with $Y_0$ 
carry the same information as the observation $(Y_s)_{s\in[0,t]}$, 
i.e., that the $\sigma$-algebra generated by $(\bar V_s)_{s\in[0,t]}$ 
and $Y_0$ coincides with the $\sigma$-algebra 
generated by $(Y_s)_{s\in[0,t]}$ for every $t$. 
An affirmative result concerning this conjecture,  
under quite general conditions on the 
filtering models (but without jump components) was proved 
in \cite{Krylov1979} and \cite{HL2008}. For our filtering model we conjecture 
that $(\bar V_s)_{s\in[0,t]}$, together with $Y_0$ and 
$\{\tilde N((0,s]\times \Gamma):s\in[0,t], \Gamma\in\cZ_1\}$ carry the same 
information as the observation $(Y_s)_{s\in[0,t]}$, 
if Assumption \ref{assumption SDE}  
holds and the coefficients of \eqref{system_1} satisfy an appropriate Lipschitz condition.
\end{remark}

We will prove Theorem \ref{theorem Z1} by deducing 
equation \eqref{eqZ2} from equation 
\eqref{eqZ1}, which we obtain by taking, under $Q$, 
the conditional expectation of the terms in the equation 
for $\gamma_t^{-1}\varphi(X_t)$,  given the observation $\{Y_s: s\leq t\}$.

There are several known conditions ensuring that Assumption 
\ref{assumption Girsanov} is satisfied. 
For a simple proof for 
the well-known Novikov condition and Kazamaki condition, 
and their generalizations we refer to 
Exercise 6.8.VI in \cite{K},  \cite{K2002} and \cite{K2019}. 
These conditions, are clearly satisfied if $|B|$ is bounded,  
but it does not seem to be easy to reformulate them in terms 
of the coefficients of the system of equations \eqref{system_1}, 
if $|B|$ is unbounded. 
Here we give a condition, which together with Assumption \ref{assumption SDE}(i) 
ensures that Assumption \ref{assumption Girsanov} holds. 

\begin{assumption}                                                        \label{assumption Girsanov 2}
There is a constant $K$ such that
$$
-x^i\rho^{ik}(t,z)B^{k}(t,z)\leq K(1+|z|^2)
\quad 
\text{for all $t\in[0,T]$, $z=(x,y)\in\bR^{d+d'}$}.
$$
\end{assumption}
\begin{remark}
Define the $\bR^{(d+d')\times d'}$-valued function $\hat\rho$ by 
$\hat\rho^{jk}:=\rho^{jk}$ for $j=1,2,...,d$, $k=1,2,...,d'$ and 
$\hat\rho^{jk}:=0$ for $j=d+1,...,d+d'$, $k=1,2,...,d'$. Then 
Assumption \ref{assumption Girsanov 2} means that the  
``one-sided linear growth" condition 
$$
zf(t,z)\leq K(1+|z|^2), 
\quad
t\in[0,T],\, z\in\bR^{d+d'}, 
$$ 
holds for the $\bR^{d+d'}$-valued function 
$f=-\hat\rho B$, where 
$zf$ denotes the standard inner product of the vectors 
$z,f\in\bR^{d+d'}$. 
Clearly, this condition is essentially 
weaker then the linear growth 
condition on $f$ (in $z\in\bR^{d+d'}$), 
which obviously holds if one of the functions 
$\rho$ and $B$ is bounded in magnitude 
and the other satisfies the linear growth condition in Assumption 
\ref{assumption SDE} (i). 
\end{remark}
\begin{theorem}                                                               \label{theorem Girsanov}
Let Assumptions \ref{assumption SDE}(i)  
and \ref{assumption Girsanov 2} hold. Then $\E\gamma_T=1$, i.e., 
Assumption \ref{assumption Girsanov} holds. 
\end{theorem}
\begin{proof}
We want to prove $\E({\bf1}_{|Z_0|\leq R}\gamma_T)=P(|Z_0|\leq R)$ for every constant $R>0$, since 
by monotone convergence it implies 
$$
\E\gamma_T=\lim_{R\to\infty}\E({\bf1}_{|Z_0|\leq R}\gamma_T)=\lim_{R\to\infty}P(|Z_0|\leq R)=1.
$$
To this end we fix a constant $R>0$  and set $\bar\gamma_t:={\bf1}_{|Z_0|\leq R}\gamma_t$. 
By It\^o's formula 
$$
d\bar\gamma_t=-\bar\gamma_t B(t,Z_t)\,dV_t, 
$$
that shows that 
$\bar{\gamma}$ is a local $\cF_t$-martingale.  
Thus $\E\bar\gamma_{T\wedge \tau_n}=P(|Z_0|\leq R)$ 
for an increasing sequence $(\tau_n)_{n=1}^{\infty}$ of stopping times $\tau_n$ 
such that $\tau_n$ converges to $\infty$ as $n\to\infty$, 
and $(\gamma_{t\wedge\tau_n})_{t\in[0,T]}$ is a martingale for every $n$.
 Consequently, if we can show $\E\sup_{t\leq T}\bar\gamma_t<\infty$, then we can use Lebesgue's 
theorem on dominated convergence to get $\E\bar\gamma_T=P(|Z_0|\leq R)$.  
Define the stopping times 
$$
\tau_n=\inf\{t\in[0,T]: [\bar\gamma]_t\geq n\}\quad\text{for integers $n\geq1$}, 
$$
where
$$
[\bar{\gamma}]_t=\int_0^t\bar{\gamma}_s^2|B(s,Z_t)|^2\,ds. 
$$
Then by standard estimates, using the Davis inequality, we have 
\begin{align}
\E\sup_{t\leq T}\bar\gamma_{t\wedge\tau_n}
\leq 1+3\E[\bar\gamma]^{1/2}_{T\wedge\tau_n}                                                                             
\leq &1+3\E\sup_{t\leq T}\bar\gamma^{1/2}_{t\wedge\tau_n}
\Big(\int_0^{T\wedge\tau_n}          
\bar\gamma_t|B(t,Z_t)|^2\,dt
\Big)^{1/2}                                                                             \nonumber\\
\leq &1+\tfrac{1}{2}\E\sup_{t\leq T}\bar\gamma_{t\wedge\tau_n}
+5\E\int_0^T\bar\gamma_t|B(t,Z_t)|^2\,dt,                                                                               \nonumber
\end{align}
which, after we subtract 
$\tfrac{1}{2}E\sup_{t\leq T}\bar\gamma_{t\wedge\tau_n}$ and let $n\to\infty$, 
by Fatou's lemma gives
$$
\tfrac{1}{2}\E\sup_{t\leq T}\bar\gamma_{t}
\leq 1+5\E\int_0^T\bar\gamma_t|B(t,Z_t)|^2\,dt
\leq 1+5\E\int_0^T\bar\gamma_t(K_0+K_2|Z_t|^2)\,dt.  
$$
Since $\E\bar\gamma_t\leq 1$, to show that the right-hand side of the last inequality is finite 
we need only prove that if $K_2\neq0$ then
\begin{equation}                                                                                                      \label{gmoment}
\sup_{t\leq T}\E\bar\gamma_t |Z_t|^2<\infty. 
\end{equation}
To this end we apply It\^o's formula 
to $U_t:=\bar\gamma_t|Z_t|^2$ and use 
Assumptions \ref{assumption SDE} (ii) and \ref{assumption Girsanov 2} to get 
\begin{align}                                                           
dU_t=&\bar\gamma_t(2X_tb(t,Z_t)+2Y_tB(t,Z_t)
+|\sigma(t,Z_t)|^2+|\rho(t,Z_t)|^2+1)\,dt                                                           \nonumber\\
&-2\bar\gamma_t(X_t\rho(t,Z_t)B(t,Z_t)
+Y_tB_t(t,Z_t))\,dt
+\bar\gamma_t\int_{\frZ_0}|\eta(t,Z_t,\frz)|^2\,\nu_0(d\frz)dt  \nonumber\\
&+\bar\gamma_t\int_{\frZ_1}|\xi(t,Z_t,\frz)|^2\,\nu_1(d\frz)dt
+\bar\gamma_t\int_{\frZ_1}|\frz|^2\,\nu_1(d\frz)dt+dm_t             \nonumber\\
\leq& N\bar\gamma_t\,dt +NU_t\,dt+dm_t
\end{align}
with a constant $N$ and a cadlag local martingale $m$ starting from zero. 
Hence by a standard stopping time argument and Gronwall's 
inequality we get a constant $N$
such that 
$$
\sup_{t\leq T}\E U_{t\wedge\tau_n}\leq N(1+\E({\bf1}_{|Z_0|\leq R}|Z_0|^2))<\infty
$$
for an increasing sequence of stopping times $\tau_n\uparrow\infty$. 
Letting here $n\to\infty$ by Fatou's lemma 
we get \eqref{gmoment}, which finishes 
the proof of the theorem. 
\end{proof}

\mysection{Preliminaries}                                                               \label{section preliminaries}

The following lemma is our main tool 
for calculating conditional expectations of 
Lebesgue and It\^o stochastic 
integrals of simple processes under $Q$ given $\cF^Y_t$.
\begin{lemma}                                                                             \label{lemma p1}
Let $X$ and $Y$ be random variables such that 
$\E|X|<\infty$, $\E|Y|< \infty$ and $\E|XY|<\infty$. Let $\cG^1$, $\cG^2$ 
and $\cG$ be $\sigma$-algebras of events such that $\cG^1\subset\cG$, 
$\cG^2$ is independent of $\cG$, $X$ is $\cG$-measurable and 
$Y$ is independent of $\cG\vee\cG^2:=\sigma(\cG,\cG^2)$. 
Then 
$$
\E(XY|\cG^1\vee\cG^2)=\E(X|\cG^1)\E Y. 
$$
\end{lemma}

\begin{proof}
The right-hand side of the above equation is a $\cG^1$-measurable random variable, hence it is obviously 
$\cG^1\vee\cG^2$-measurable. Let $\cH$ denote the family of  
$G\in\cG^1\vee\cG^2$ such that 
$$
\E Y\E(\E(X|\cG^1){\bf1}_G)=\E(XY{\bf1}_G).
$$
Then $\cH$ is a $\lambda$-system, and for $G=G_1\cap G_2$, $G_i\in\cG^i$ 
we have 
$$
\E Y \E(\E(X|\cG^1){\bf1}_G)=
\E Y\E(\E({\bf1}_{G_1}X|\cG^1){\bf1}_{G_2}))=
\E Y\E(\E({\bf1}_{G_1}X|\cG^1))\E{\bf1}_{G_2}
$$
$$
=\E Y\E({\bf1}_{G_1}X)\E{\bf1}_{G_2}=\E(XY{\bf1}_{G}), 
$$
that shows that $\cH$ contains the $\pi$-system $\{G_1\cap G_2:G_i\in\cG^i, i=1,2\}$. 
Hence, by Dynkin's monotone class lemma $\cH=\cG^1\vee\cG^2$, 
which completes the proof. 
\end{proof}
To formulate a theorem on conditional expectations of Lebesgue and 
It\^o integrals we consider a complete filtered probability space 
$(\Omega,\cF,P,\cF_t)$ carrying independent 
$\cF_t$-Wiener processes $W^{i}=(W^{i}_t)_{t\geq0}$ 
and independent $\cF_t$-Poisson random 
measures $N_i=N_i(\frz,dt)$ with $\sigma$-finite 
characteristic measures $\nu_i$ on separable 
measurable spaces $(\frZ_i,\cZ_i)$ for $i=0,1$, 
respectively.  We denote by $\cG_t$ the $P$-completion 
of the 
$\sigma$-algebra generated by the events of 
a $\sigma$-algebra $\cY_0\subset\cF_0$ 
together with the random variables 
$W^1_s$ and $N_1((0,s]\times\Gamma)$ 
for $s\leq t$ and $\Gamma\in\cZ_{1}$ 
such that $\nu_1(\Gamma)<\infty$.  
The predictable $\sigma$-algebras on $\Omega\times[0,T]$, 
relative to $(\cF_t)_{t\geq0}$ 
and $(\cG_t)_{t\geq0}$ are denoted by $\cP_{\cF}$ and 
$\cP_{\cG}$, respectively. The optional $\sigma$-algebras relative to 
$(\cF_t)_{t\geq0}$ 
and $(\cG_t)_{t\geq0}$ are denoted by $\cO_{\cF}$ and 
$\cO_{\cG}$, respectively.

The following definition will be frequently used.
\begin{definition}
Given a probability space $(\Omega,\cF,P)$ and 
a sub-$\sigma$-algebra $\cG\subset\cF$, 
we say that a random variable $f$ is $\sigma$-integrable 
(with respect to $P$) relative to $\cG$, if there exists an increasing 
sequence $(\Omega_n)_{n=1}^{\infty}$ 
such that $\bigcup_n \Omega_n = \Omega$, $\Omega_n\in\cG\!$  
and $\E |f{\bf1}_{\Omega_n}|<\infty$ for all $n$.
\end{definition}

One knows that  for nonnegative random variables $f$ 
and $\sigma$-algebras $\cG\subset\cF$ the conditional 
expectation $\E(f|\cG)$ is well-defined, and that for general 
random variables $f$ the extended conditional expectation 
$\E(f|\cG)$ 
is defined as $\E(f^+|\cG)-\E(f^-|\cG)$ on the set $\E(|f||\cG)<\infty$ 
and it is defined to be $+\infty$ on $\E(|f||\cG)=\infty$. 
It is not difficult to show that the extended conditional 
expectation $\E(f|\cG)$ is almost surely finite if and only if $f$ 
is $\sigma$-integrable (with respect to $P$) relative to $\cG$.

We  consider real-valued $\cF\otimes\cB([0,T])$-measurable 
$\cF_t$-adapted processes $f=(f_t)_{t\in[0,T]}$ and $g=(g_t)_{t\in[0,T]}$ 
on $\Omega\times[0,T]$, 
real-valued $\cF\otimes\cB([0,T])\otimes\cZ_i$-measurable functions  
$h^{(i)}=h^{(i)}_t(\omega,\frz)$ of 
$(\omega,t,\frz)\in\Omega\times[0,T]\times \frZ_i$ 
for $i=0,1$, and 
a real-valued $\cF\otimes\cB([0,T])\otimes\cZ$-measurable function 
$h=h_t(\omega,\frz)$ of 
$(\omega,t,\frz)\in\Omega\times[0,T]\times\frZ$, 
such that for every $t\in[0,T]$ the functions 
$h^{(i)}_t$ and $h_t$ are $\cF_t\otimes\cZ_i$-measurable 
and $\cF_t\otimes\cZ$-measurable, respectively, 
for $i=0,1$, 
where $(\frZ,\cZ)$ is a separable measurable space, 
equipped with a $\sigma$-finite measure 
$\nu$.
Assume that almost surely
\begin{equation}                                                                                      \label{fh}
F:=\left(\int_0^T |f_s|^2\,ds\right)^{1/2}<\infty
\quad 
H^{(i)}
:=\left(\int_0^T \int_{\frZ_i}|h^{(i)}_s(\frz)|^2\,\nu_i(d\frz) ds\right)^{1/2}<\infty
\end{equation}
\begin{equation}                                                                                        \label{gh}
G:=\int_0^T|g_s|\, ds<\infty, 
\quad
H:=\int_0^T \int_{\frZ}|h_s(\frz)|\,\nu(d\frz)ds<\infty 
\end{equation}
for $i=0,1$.  Then the processes   
$$
\alpha_t:=\int_0^tg_s\,ds, \quad \delta_t
:=\int_0^t\int_{\frZ}h_s(\frz)\,\nu(d\frz)ds,\quad t\in[0,T], 
$$
and 
\begin{equation}                                                                                             \label{abc}
\beta_t^{(i)}=\int_0^tf_s\,dW^{i}_s, \quad 
\delta^{(i)}_t=\int_0^t\int_{\frZ_i}h^{(i)}_s(\frz)\,\tilde N_i(d\frz,ds), 
\quad t\in[0,T],
\end{equation}
are well-defined for $i=0,1$,  and we have the following theorem.
    
\begin{theorem}                                                                                                  \label{theorem p0}
Assume the random variables 
$F^r$, $G$, $H$ and $|H^{(i)}|^2$ 
for $i=0,1$, for some $r>1$ are $\sigma$-integrable 
(with respect to $P$) relative to $\cG_0$. 
Then for $t\in[0,T]$ we have 
\begin{equation}                                                                                    
\E(\beta^{(1)}_t| \cG_t) = \int_0^t \hat f_s\,dW^1_s,   
 \quad   
 \E(\beta_t^{(0)}| \cG_t)= 0,                                                                \label{p1}
\end{equation}
\begin{equation}
 \E(\alpha_t|\cG_t)=\int_0^t\hat g_s\,ds,  
 \quad
 \E(\delta_t|\cG_t)=\int_0^t\int_{\frZ}\hat h_s(\frz)\,\nu(d\frz)ds,                             \label{p3}
\end{equation}
\begin{equation}
 \E(\delta^{(1)}_t|\cG_t)     
 =\int_0^t\int_{\frZ_1}\hat h^{(1)}_s(\frz)\,\tilde N_1(d\frz,ds),   \quad 
 \E(\delta^{(0)}_t|\cG_t)=0                                                                  \label{p4}                                       
 \end{equation}
almost surely for some $\cP_{\cG}$-measurable functions $\hat f$ and $\hat g$ on 
$\Omega\times[0,T]$, a $\cP_{\cG}\otimes\cZ_1$-measurable function 
$\hat h^{1}$ on $\Omega\times[0,T]\times\frZ_1$, and a  
$\cP_{\cG}\otimes\cZ$-measurable function 
$\hat h$ on $\Omega\times[0,T]\times\frZ$
such that 
\begin{equation}                                                                        \label{hatfg}                                                                          
\hat f_t=\E(f_t| \cG_t), \quad \hat g_t=\E(g_t|\cG_t) 
\quad\text{\rm{(a.s.)} for $dt$-a.e. $t\in[0,T]$}, 
\end{equation}
\begin{equation}                                                                           \label{hath1}
\hat h^{(1)}_t=\E(h^{(1)}_t(\frz)|\cG_t)
\quad
\text{\rm{(a.s.)} for $dt\otimes \nu_1$-a.e. 
$(t,\frz)\in[0,T]\times\frZ_1$},
\end{equation}
\begin{equation}                                                                           \label{hath}
\hat h_t=\E(h_t(\frz)|\cG_t)\quad
\text{\rm{(a.s.)} for $dt\otimes \nu$-a.e. $(t,\frz)\in[0,T]\times\frZ$}. 
\end{equation}
\end{theorem}

\begin{proof}
Since $F^r$ is $\sigma$-integrable with respect to $\cG_0$,  
there is an increasing sequence 
$\Omega_n\in\cG_0$ such that  
$\bigcup_{n=1}^{\infty}\Omega_n=\Omega$ and 
$\E({\bf1}_{\Omega_n}F^r)<\infty$ 
for every integer $n\geq1$. 
By the definition and elementary properties of (extended) conditional expectations 
and stochastic integrals, we have 
$$
 {\bf1}_{\Omega_n}\E\Big(\int_0^tf_s\,dW^i_s\Big| \cG_t\Big)= 
\E\Big( {\bf1}_{\Omega_n}\int_0^tf_s\,dW^i_s\Big| \cG_t\Big)=
\E\Big(\int_0^t {\bf1}_{\Omega_n}f_s\,dW^i_s\Big| \cG_t\Big),
$$
$$ 
{\bf1}_{\Omega_n}\E(f_t|\cG_t)=\E({\bf1}_{\Omega_n}f_t|\cG_t), \quad t\in[0,T]
$$
for $i=0,1$ and every $n\geq1$. Thus, 
taking ${\bf1}_{\Omega_n}f$ in place of $f$, 
we may assume that $\E F^r<\infty$. Similarly, 
we may also assume that $\E G$, $\E H$ and 
$\E|H^{(i)}|^2$ are 
finite in what follows below.
Assume first that $f$ belongs to $\cH_0$, 
the set of simple processes of the form 
\begin{equation}                                                                                              \label{s1}
    f_t = \sum_{i=0}^{k-1}\xi_i \1_{(t_i,t_{i+1}]}(t), 
\end{equation}
where $0=t_0\leq \cdots \leq t_k = T$ are deterministic time instants,  
and $\xi_i$ is a bounded  $\F_{t_i}$-measurable 
random variable for every $i=0,1,...,k-1$ 
for an integer $k\geq1$.  Then we have 
\begin{equation}                                                                                                \label{p0}
\E\Big(\int_0^tf_s\,dW^1_s\Big| \cG_t\Big)
=\sum_{i}\E\big(\xi_i(W^1_{t_{i+1}\wedge t}-W^1_{t_{i}\wedge t})\big|\cG_t\big),\quad 
\text{for $t\in[0,T]$}.
\end{equation}
For $0\leq r\leq s\leq T$ define the $\sigma$-algebra 
$$
\cG_{r,s}
=\sigma
(W^1_v-W^1_u, N_1(\Gamma\times(u,v]):\,  
r\leq u\leq v\leq s, \Gamma\in\cZ_1, \nu_1(\Gamma)<\infty). 
$$
Then $\sigma$-algebras $\cG_{r}$ and $\cG_{r,s}$ 
are independent and $\cG_s=\cG_{r}\vee\cG_{r,s}$. 
Thus, using Lemma \ref{lemma p1} 
with $X:=\xi_i$, $Y:=1$, 
$\cG^1:=\cG_{t_i}$, 
$\cG:=\cF_{t_i}$
and $\cG^2:=\cG_{t_i,s}$ for $t_i\leq s\leq T$,  
we have 
\begin{equation}                                                                         \label{conditioning}                                                                  
\E(\xi_i|\cG_s)=\E(\xi_i|\cG_{t_i})
\quad
\text{for $i=0,1,2,...,k-1$}.
\end{equation}
Hence for $t_i\leq s\leq t_{i+1}\leq t\leq T$,
\begin{equation}                                                                                     \label{p1W}
\E(\xi_i(W^1_{t_{i+1}}-W^1_{t_{i}})|\cG_{t})=
\E(\xi_i|\cG_t)(W^1_{t_{i+1}}-W^1_{t_{i}})=
\E(\xi_i|\cG_s)(W^1_{t_{i+1}}-W^1_{t_{i}})  
\end{equation}
and for $t_j\leq s\leq t\leq T$, 
\begin{equation}                                                                                    \label{p2W}
\E(\xi_j(W^1_t-W^1_{t_j})|\cG_t)
=\E(\xi_j|\cG_{t})(W^1_{t}-W^1_{t_j})
=\E(\xi_j|\cG_{s})(W^1_{t}-W^1_{t_j}). 
\end{equation}
Consequently, defining $\hat f_s=\E(\xi_i|\cG_s)=\E(f_s|\cG_s)$ 
for $s\in(t_i,t_{i+1}]$, $i=0,1,...,k-1$, 
the function $\hat f$ on $\Omega\times[0,T]$ is $\cP_{\cG}$-measurable, 
and using \eqref{p0} 
we can see that the first equation in \eqref{p1} holds. 
Assume now that $f$ is 
$\cF\otimes\cB([0,T])$-measurable and $\cF_t$-adapted
such that $\E F^r<\infty$.  
Then there are sequences $(f^n)_{n=1}^{\infty}$ and $(\hat f^n)_{n=1}^{\infty}$ 
such that $f^n\in\cH_0$, $\hat f ^n$ is $\cP_{\cG}$-measurable,  
\begin{equation}                                                                                  \label{limf}
\lim_{n\to\infty}\E\left(\int_0^T|f_t-f^n_t|^2\,dt\right)^{r/2}=0, 
\end{equation}
and almost surely 
\begin{equation}                                                                        \label{project}
\E(I_t(f^n)|\cG_t):=\E\Big(\int_0^t f^n_s\,dW^1_s\Big|\cG_t\Big)
=\int_0^t \hat f^n_s\,dW^1_s
=:I_t(\hat f^n)\quad\text{for all $t\in[0,T]$, }
\end{equation}
\begin{equation}                                                                                         \label{6.23.1}
\hat f^n_t=\E(f^n_t|\cG_t)\quad \text{for $dt$-a.e. $t\in[0,T$]} 
\end{equation}
for all $n\geq1$.
Using the Davis inequality, Doob's inequality, 
Jensen's and Burkholder's inequalities for any $r>1$ we have  
$$
\E\left(\int_0^T|\hat f^n_t-\hat f^m_t|^2\,dt\right)^{1/2}
\leq 3 \E\sup_{t\leq T} |I_t(\hat f^n-\hat f^m)|
$$
$$
=3\E\sup_{t\in[0,T]\cap\bQ}|\E \big(I_t(f^n-f^m)\big|\cG_t\big)|
\leq 3\E\sup_{t\in[0,T]\cap\bQ}(\E(\sup_{s\leq T}|I_s(f^n-f^m)||\cG_t)
$$
$$
=3\frac{r}{r-1}\left(\E\sup_{t\leq T}|I_t(f^n-f^m)|^r\right)^{1/r}
\leq N\left(\E\left(\int_0^T|f^n_t-f^m_t|^2\,dt\right)^{r/2}\right)^{1/r}, 
$$
where $\bQ$ is the set of rational numbers 
and $N=N(r)$ is a constant, 
which gives 
$$
\lim_{n,m\to\infty}
\E\left(\int_0^T|\hat f^n_t-\hat f^m_t|^2\,dt\right)^{1/2}=0. 
$$
Thus there exists a $\cP_{\cG}$-measurable function 
$\hat f$ on $\Omega\times[0,T]$, 
such that 
\begin{equation}                                                                      \label{lim0}
\lim_{n\to\infty}\E\left(\int_0^T|\hat f_t-\hat f^n_t|^2\,dt\right)^{1/2}=0,  
\end{equation}
which implies 
\begin{equation}                                                                       \label{lim1}
\lim_{n\to\infty}\E\sup_{t\in[0,T]}|I_t(\hat f)-I_t(\hat f^n)|=0. 
\end{equation}
Using Jensen's and Davis' inequalities again 
we have
$$
\E|\E(I_t(f)|\cG_t)-\E(I_t(f^n)|\cG_t)|
\leq
\E\E(|I_t(f-f^n)||\cG_t)
$$
$$
=\E |I_t(f-f^n)|
\leq 3\E\left(\int_0^T|f_t-f^n_t|^2\,dt\right)^{1/2}\quad\text{for every $t\in[0,T]$},  
$$
i.e., for $n\to\infty$ 
\begin{equation}                                                                        \label{lim2}
\E(I_t(f^n)|\cG_t)\to \E(I_t(f)|\cG_t)
\quad 
\text{in $L_1(\Omega)$ for every $t\in[0,T]$}. 
\end{equation}
Thus letting $n\to\infty$ in equation \eqref{project}, by virtue of 
\eqref{lim1} and  \eqref{lim2} we get the first equation  
in \eqref{p1}. 
Clearly,   \eqref{limf} and \eqref{lim0} imply
$$
\lim_{n\to\infty}\int_0^T\E|f_t-f^n_t|+\E|\hat f_t-\hat f^n_t|\,dt=0.
$$
Hence there is a subsequence $n_l\to\infty$ and  a set 
$S\in\cB([0,T])$ of Lebesgue measure $0$ such that for $n_l\to\infty$,
$$
\text{$f^{n_l}_t\to f_t$ and a 
$\hat f^{n_l}_t\to \hat f_t$ in $L_1(\Omega)$ for each $t\in [0,T]\setminus S=:S^c$},  
$$
and taking into account \eqref{6.23.1}, we can assume that $S$ is a $dt$-zero set such that 
we also have $\hat{f}_t^{n_l}=\E(f^{n_l}_t|\cG_t)$ (a.s.) for every $t\in S^c$. 
 Thus  for  $n_l\to\infty$ we have $\E(f^{n_l}_t|\cG_t)\to \E(f_t|\cG_t)$ 
in $L_1(\Omega)$ for each $t\in S^c$, which gives 
$$
\hat f_t=\E(f_t|\cG_t)\quad \text{almost surely for every $t\in S^c$},  
$$
i.e., the first equation in \eqref{hatfg} holds. To prove the second equation in 
\eqref{p1} we note that 
for $\xi_i$ from the 
expression \eqref{s1}  we have 
\begin{equation}                                                                                     \label{pW2}
\E(\xi_i(W^0_{t_{i+1}\wedge t}-W^0_{t_i})|\cG_t)
=\E(\xi_i|\cG_{t_i})\E(W^0_{t_{i+1}\wedge t}-W^0_{t_i})=0
\quad 
\text{for $i=1,2,...,N-1$}, 
\end{equation}
by using Lemma \ref{lemma p1} with $X=\xi_i$, 
$Y=W^0_{t_{i+1}\wedge t}-W^0_{t_{i}\wedge t}$ 
$\cG^1:=\cG_{t_i}\subset\cF_{t_i}=:\cG$ and 
$\cG^2:=\cG_{t_i,t}$ for $t_i\leq t$. 
Hence we get the second equation in \eqref{p1} 
for $f$ given in \eqref{s1}, 
and the general case follows 
by approximation as above.
To prove the first equation in 
\eqref{p3} assume that $g$ is given by the right-hand side of \eqref{s1}. 
Then using \eqref{conditioning} we can see that  
$$
\hat g_t:=\sum_{i=0}^{k-1}
\E(\xi_i|\cG_{t_i}){\bf1}_{(t_i,t_{i+1}]}(t)
=\E(g_t|\cG_t), \quad t\in[0,T], 
$$
and that the first equation in \eqref{p3} 
and the second equation in \eqref{hatfg} hold.
Assume now that $g$ is an $\cF\otimes\cB([0,T])$-measurable 
$F_t$-adapted random process such that $\E G<\infty$.  
Then there are sequences $(g^n)_{n=1}^{\infty}$ 
and $(\hat g^n)_{n=1}^{\infty}$ 
such that $g^n\in\cH_0$, $\hat g ^n$ is $\cP_{\cG}$-measurable,  
\begin{equation*}                                                                                  \label{limg}
\lim_{n\to\infty}\E\int_0^T|g_t-g^n_t|\,dt=0, 
\end{equation*}
and almost surely 
\begin{equation*}                                                                        \label{projectg}
\E\Big(\int_0^t g^n_s\,ds\Big|\cG_t\Big)
=\int_0^t \hat g^n_s\,ds
\quad\text{for all $t\in[0,T]$, }
\end{equation*}
\begin{equation*}                                                                                         \label{1.2.5.22}
\hat g^n_t=\E(g^n_t|\cG_t)\quad \text{for $dt$-a.e. $t\in[0,T$]} . 
\end{equation*}
Hence noting that by Tonelli's theorem and Jensen's inequality 
$$
\E\int_0^T|\hat g^n_t-\hat g^m_t|\, dt
=\int_0^T\E|\E(g^n_t|\cG_t)-\E(g^m_t|\cG_t)|\,dt
$$
$$
\leq \int_0^T\E\E(|g^n_t-g^m_t||\cG_t)\,dt=\E\int_0^T|g^n_t-g^m_t|\,dt, 
$$
and repeating previous arguments we get 
a $\cP_{\cF}$-measurable $\hat g$ 
such that the first equation in \eqref{p3} and the second equation 
in \eqref{hatfg} hold. To prove the second equation in \eqref{p3} 
we assume first that 
\begin{equation}                                                                              \label{s3}
h_t(\frz)=\sum_{i=0}^{k-1}
\xi_i{\bf1}_{(t_i,t_{i+1}]\times\Gamma_i}(t,\frz), 
\end{equation}
for a partition $0\leq t_0\leq t_1\leq...\leq t_k=T$ of $[0,T]$, 
bounded $\cF_{t_i}$-measurable 
random variables $\xi_i$ and sets 
$\Gamma_i\in\cZ$, $\nu(\Gamma_i)<\infty$ for $i=0,...,k-1$.
Then
$$ 
\E\Big(\int_0^t \int_{\frZ}h_s(\frz)\,\nu(d\frz)\,ds\Big|\cG_t  \Big)
=\sum_{i=0}^{k-1}
\E(\xi_i|\cG_t)\nu(\Gamma_i)(t_{i+1}\wedge t-t_{i}\wedge t).  
\quad t\in[0,T],
$$
Thus, since by virtue of \eqref{conditioning} we have 
$$
\hat h_t(z):=\sum_{i=0}^{k-1}
\E(\xi_i|\cG_{t_i}){\bf1}_{(t_i,t_{i+1}]}(t){\bf1}_{\Gamma_i}(\frz)
=\E(h_t(\frz)|\cG_t), \quad t\in[0,T], \,\frz\in \frZ,
$$
for $\hat h$ the second equation in \eqref{p3} and by definition \eqref{hath} hold. 
Hence we can get these equations in the general case by a straightforward approximation procedure 
in the same way as the first equation in \eqref{p3} and the second equation in 
\eqref{hatfg} have been proved above. 

Now we are going to prove \eqref{p4}. 
Assume first that $h^{(1)}$ is a simple 
function, given by the right-hand side of equation 
\eqref{s3} with $\Gamma_i\in\cZ_1$, 
$\nu_1(\Gamma_i)<\infty$, $i=0,1,...,k-1$. 
Then
$$ 
\E\Big(\int_0^t \int_{\frZ_1}h^{(1)}_s(\frz)\,\tilde N_1(d\frz,ds)\Big|\cG_t  \Big)
=\sum_{i=0}^{k-1}
\E\big(\xi_i\tilde N_1(\Gamma_i\times(t_i\wedge t,t_{i+1}\wedge t])\big|\cG_t\big). 
$$
In the same way as equations \eqref{p1W} and \eqref{p2W} are obtained,  
by using \eqref{conditioning} 
we get 
$$
\E\big(\xi_i\tilde N_1( \Gamma_i\times(t_i,t_{i+1}])\big|\cG_t\big)
=\E(\xi_i|\cG_{t_i})\tilde N_1( \Gamma_i\times(t_i,t_{i+1}])
=\E(\xi_i|\cG_{s})\tilde N_1( \Gamma_i\times(t_i,t_{i+1}])
$$
for $t_{i}\leq s\leq t_{i+1}\leq t$, and 
$$
\E\big(\xi_j\tilde N_1( \Gamma_j\times(t_j,t])\big|\cG_t\big)
=
\E(\xi_j|\cG_{t_j})\tilde N_1( \Gamma_j\times(t_j,t])=
\E(\xi_j|\cG_{s})\tilde N_1( \Gamma_j\times(t_j,t])
$$
for $t_{j}\leq s\leq t\leq t_{j+1}$. Thus 
for 
$$
\hat h_t(\frz)=\sum_{i=0}^{k-1}
\E(\xi_{i}|\cG_{t_i}){\bf1}_{(t_i,t_{i+1}]\times\Gamma_i}(t,\frz)
=\E(h_t(\frz)|\cG_t),
$$
equations in \eqref{p4} and \eqref{hath1} hold. 
Assume now that $h^{(1)}$ is $\cF\otimes\cB([0,T])\otimes\cZ$-measurable 
such that for every $t\in[0,T]$ the function $h_t^{(1)}$ is 
$\cF_t\otimes\cZ_1$-measurable and $\E|H^{(1)}|^2<\infty$, 
where $H^{(1)}$ is defined in \eqref{fh}. 
Then there exist sequences 
$(h^n)_{n=1}^{\infty}$ and $(\hat h^n)_{n=1}^{\infty}$, such that $h^n$ 
is a simple function of the form \eqref{s3}, $\hat h^n$ 
is a $\cP_{\cG}\otimes\cZ_1$-measurable function,
\begin{equation}                                                                                               \label{phn}
\E(\tilde I_t(h^n)|\cG_t)
:=\E\Big(
\int_0^t\int_{\frZ_1}h^n_s(\frz)\,\tilde N_1(d\frz,ds)\big|\cG_t
\Big)
=\int_0^t\int_{\frZ_1}\hat h^n_s(\frz)\,\tilde N_1(d\frz,ds), 
\end{equation}
\begin{equation}                                                                          \label{hathn}
\hat h^n_t(z)=\E(h^n_t(z)|\cG_t),
\quad
\text{almost surely, 
for $\nu_1(d\frz)\otimes dt$-a.e. $(\frz,t)\in\frZ_1\times[0,T],$}
\end{equation}
for every $n\geq1$, and 
\begin{equation}                                                                                               \label{limhn}
\lim_{n\to\infty}
\E\int_0^T\int_{\frZ_1}|h^{(1)}_t(\frz)-h^n_t(\frz)|^2\,\nu_1(d\frz)\,dt=0. 
\end{equation}
Hence using Jensen's inequality we get 
$$
\lim_{n,m\to \infty}\E\int_0^T\int_{\frZ_1}|\hat h^n_t(\frz)-\hat h^m_t(\frz)|^2\,\nu_1(d\frz)dt=0, 
$$
which implies the existence of a 
$\cP_{\cG}\otimes\cZ_1$-measurable function 
$\hat h^{(1)}$ such that 
\begin{equation}                                                                                                     \label{limhathn}
\lim_{n\to\infty}\E\int_0^T\int_{\frZ_1}|\hat h^{(1)}_t(\frz)-\hat h^n_t(\frz)|^2\,\nu_1(d\frz)\,dt=0. 
\end{equation}
Thus letting $n\to\infty$ in \eqref{phn} we obtain \eqref{p4}. By virtue of \eqref{limhn} 
and \eqref{limhathn} there is a subsequence $n_l\to\infty$ and a set 
$A\in\cB([0,T])\otimes\cZ_1$ such that $dt\otimes\nu_1(A)=0$ 
and for $n_l\to\infty$ 
$$
h^{n_l}_t(\frz)\to  h^{(1)}_t(\frz)
\quad
\text{and\,\, $\hat h^{n_l}_t(\frz)
\to  \hat h^{(1)}_t(\frz)$ in 
mean square} 
$$
for every $(t,\frz)\in A^c:=[0,T]\times\frZ_1\setminus A$. Consequently, 
$$
\E(h^{n_k}_t(\frz)|\cG_t)\to \E(h^{(1)}_t(\frz)|\cG_t) 
\quad\text{in mean square for every $(\frz,t)\in A^c$}, 
$$
and letting $n:=n_l\to\infty$ in \eqref{hathn} 
we obtain $\E(h^{(1)}_t(\frz)|\cG_t)=\hat h^{(1)}_t(\frz)$ 
for $(\frz,t)\in A^c$, 
which proves \eqref{hath1}. 
To prove the second equation in \eqref{p4} assume first that $h^{(0)}$ 
is a simple function of the form \eqref{s3} with $\Gamma_i\in\cZ_0$, 
$\nu_0(\Gamma_i)<\infty$ for $i=0,1,...k-1$. 
Just like \eqref{pW2} is obtained, by Lemma \ref{lemma p1} we get 
$$
\E(\eta_i\tilde N_0(\Gamma_i\times(t_i\wedge t,t_{i+1}\wedge t])|\cG_t)=
\E(\eta_i|\cG_{t_i})\E\tilde N_0(\Gamma_i\times(t_i\wedge t,t_{i+1}\wedge t])=0
$$
for $i=0,1,...,k-1$ and $t\in[0,T]$, that implies the second equation in \eqref{p4}.  
Hence, we obtain the second equation in \eqref{p4} 
for $\cO_{\cF}$-measurable functions 
satisfying \eqref{fh} by approximation with simple functions. 
\end{proof}

We can reformulate the above theorem by using the notion of optional projection of 
processes with respect to a given filtration. 
It is well-known (see for instance \cite[Thm 5.1]{HWY}, \cite[Thm 2.43]{DM}) 
that if $f=(f_t)_{t\in[0,T]}$ is a $\mathcal{B}([0,T])\otimes \F$-measurable process 
such that $f_{\tau}$ is $\sigma$-integrable (with respect to 
a probability $P$) relative to the $\sigma$-algebra $\cG_{\tau}$ 
for every $\cG_t$-stopping time $\tau\leq T$ (with respect 
to a $P$-complete filtration $(\cG_t)_{t\in[0,T]}$), then 
there exists a unique (up to evanescence) $\cG_t$-optional process 
$^o\! f=({^o\! f}_t)_{t\in[0,T]}$
such that for every $\cG_t$-stopping time $\tau\leq T$
\begin{equation*}
    \E(f_{\tau}|\cF_{\tau}) ={^o\!f}_{\tau}\quad\text{(a.s.)}.
\end{equation*}
The process $^o\! f$ is called the optional projection of $f$ (under $P$ 
with respect to $(\cG_t)_{t\in[0,T]}$).  
If $f$ is a cadlag process such that almost surely $\sup_{t\leq T}|f_t|\leq \eta$ 
for some $\sigma$-integrable 
random variable $\eta$ with respect to $P$ relative to $\cG_0$, 
then almost surely the 
trajectories of $^o\! f$ have left and right limits at every $t\in(0,T]$ and $[0,T)$, 
respectively, 
and moreover, they are also almost surely right-continuous 
if $(\cG_t)_{t\in[0,T]}$ is right-continuous. 
One can define the extended optional projection $^o\!f$ of every nonnegative 
$\mathcal{B}([0,T])\otimes \F$-measurable process $(f_t)_{t\in[0,T]}$ 
(under $P$ with respect to $(\cG_t)_{t\in[0,T]}$)  
by setting $^o\!f:=\lim_{n\to\infty}{\bf1}_{S}{^o\!(f\wedge n)}$, 
where $S$ is the set of $(\omega,t)\in\Omega\times[0,T]$ such that 
$\lim_{n\to\infty} {^o\!(f\wedge n)_t(\omega)}$  exists, 
which may be infinite. For $\mathcal{B}([0,T])\otimes \F$-measurable 
stochastic processes  $(f_t)_{t\in[0,T]}$ one 
defines the extended  optional projection 
$^o\!f$ by  $^o\!f={^o\!(f^+)}-{^o\!(f^-)}$ on the set 
$A=\{(\omega,t)\in\Omega\otimes[0,T]: \,^o\!(f^+)+^o\!(f^-)<\infty\}$ 
and  $^o\!f=\infty$ on 
$\Omega\times[0,T]\setminus A$. Notice that 
for every $t\in[0,T]$ the extended 
conditional expectations $\E(f^+_t|\cG_t)$ and $\E(f^-_t|\cG_t)$ 
are almost surely equal to 
${^o\!(f^+)_t}$ and $^o\!(f^-)_t$, respectively. 
Thus defining the extended conditional expectation 
$\E(f_t|\cG_t)$ as $\E(f^+_t|\cG_t)-\E(f_t^-|\cG_t)$ on the event where 
$\E(f^+_t|\cG_t)+\E(f_t^-|\cG_t)<\infty$ and $\E(f_t|\cG_t)=\infty$ elsewhere,  
for each $t\in[0,T]$ we have 
$^o\!f_t=\E(f|\cG_t)$ (a.s.). Let $h=(h_t(\frz))$ be 
an $\cF\otimes\cB(\bR_+)\otimes\cZ$-measurable function 
on $\Omega\times\bR_+\times\frZ$. 
Then by the help of the Monotone Class Theorem it is not difficult to show the existence of 
an $\cO_{\cG}\otimes\cZ$-measurable function, which 
for each fixed $\frz\in \frZ$ gives the (possibly extended) $\cO_{\cG}$-optional projection  
of 
$h(\frz):=(h_t(\frz))_{t\geq0}$. We denote this function by $^o\!h$,  
and call it the (extended) $\cO_{\cG}$-optional projection of $h$.

\begin{corollary}                                                                         \label{corollary projection}
Assume the random variables 
$F$, $H^{(i)}$ and $G$, $H$, defined in  \eqref{fh} and \eqref{gh}, 
respectively, 
are $\sigma$-integrable relative to $\cG_0$ for $i=0,1$. 
Assume moreover that almost surely 
\begin{equation}                                                                                              \label{ph}
\int_0^T|^o\!f_t|^2dt<\infty,
\quad
\int_0^T\int_{\frZ_i}|^o\!h^{(i)}_t(\frz)|^2\,\nu_i(d\frz)dt<\infty
\quad
\text{for $i=0,1$}, 
\end{equation}
where $^o\!f$ 
 and $^o\!h^{(i)}$ are the (extended) $\cO_{\cG}$-optional projections of 
$f$ and $h^{(i)}$, respectively. Then for every $t\in[0,T]$ 
equations \eqref{p1}, \eqref{p3} and \eqref{p4} hold almost surely 
with the $\cO_{\cG}$-optional projections  
$^o\!f$, $^o\!g$, $^o\!h^{(i)}$ and $^o\!h$
in place of $\hat f$, $\hat g$, $\hat h^{(i)}$  and $\hat h$, 
respectively, for $i=0,1$. Moreover, there is a 
$dt$-null set $T_0\subset[0,T]$,  
a $dt\otimes\nu_1$-null set $B_0\subset[0,T]\times\frZ_1$ and  
a $dt\otimes\nu$-null set $B\subset[0,T]\times\frZ$,  such that  
\begin{enumerate}
\item[(i)] for each $t\in[0,T]\setminus T_0$ the random variable 
$|f_t|+|g_t|$ is $\sigma$-integrable relative to $\cG_0$ 
and 
\begin{equation}                                                                                                       \label{pgf} 
\E(f_t|\cG_t)={^o\!f_t}\in\bR, \,{\rm{(a.s.)}},
\quad  
\E(g_t|\cG_t)={^o\!g}_t\in\bR\, {\rm{(a.s.)}},
\end{equation}
\item[(ii)]
for each $(t,\frz)\in[0,T]\times\frZ_1\setminus B_0$ the random variable 
$|h^{(1)}_t(\frz)|$ is $\sigma$-integrable relative to $\cG_0$ and 
 \begin{equation}                                                                                                      \label{ph1}
\E(h^{(1)}_t(\frz)|\cG_t)={^o\!h_t}^{(1)}(\frz)\in\bR\, {\rm{(a.s.)}}, 
 \end{equation}
\item[(iii)]
for each $(t,\frz)\in[0,T]\times \frZ\setminus B$ the random variable 
$|h_t(\frz)|$ is $\sigma$-integrable relative to $\cG_0$, and 
 \begin{equation}                                                                                                      \label{oph}
\E(h_t(\frz)|\cG_t)={^o\!h_t}(\frz)\in\bR\,{\rm(a.s.)}.  
\end{equation}
\end{enumerate}
\end{corollary}

\begin{proof} Just like in the proof of the previous lemma 
without loss of generality 
we may and will assume that 
$F$, $G$, $H$ and $H^{(i)}$, $i=0,1$, have finite expectation. Thus by 
Minkowski's inequality and Tonelli's theorem we have 
$$
\left(\int_0^T(\E|f_t|)^2\,dt\right)^{1/2}
\leq \E\left(\int_0^T|f_t|^2\,dt\right)^{1/2}<\infty, 
\quad
\int_0^T\E|g_t|\,dt
=\E\int_0^T|g_t|\,dt<\infty,  
$$
$$
\int_0^T\int_{\frZ}\E|h_t(\frz)|\,\nu(d\frz)\,dt
= \E\int_0^T\int_{\frZ}|h_t(\frz)|\,\nu(d\frz)\,dt<\infty
$$
$$
\left(
\int_0^T\int_{\frZ_1}(\E|h_t^{(1)}(\frz)|)^2\,\nu_1(d\frz)\,dt
\right)^{1/2}
\leq \E\left(\int_0^T\int_{\frZ_1}
|h_t^{(1)}(\frz)|^2\nu_1(d\frz)\,dt\right)^{1/2}<\infty.  
$$
Therefore $\E|f_t|+\E|g_t|<\infty$
for $dt$-almost every $t\in[0,T]$,  $\E|h_t(\frz)|<\infty$
 for 
$dt\otimes\nu$-a.e. $(t,\frz)\in[0,T]\times\frZ$, 
and $\E|h^{(1)}_t(\frz)|<\infty$  for 
$dt\otimes\nu_1$-a.e. $(t,\frz)\in[0,T]\times \frZ_1$, i.e., 
we get \eqref{pgf}, 
\eqref{ph1} and \eqref{oph}.
Hence due to \eqref{hatfg} and \eqref{hath}  
we have \eqref{p3} with  $^o\!g$ and $^o\!h$ in place of $\hat g$ 
and $\hat h$, respectively. 
We also have \eqref{p1} and \eqref{p4} 
with $^o\!f$ and $^o\!{h^{(1)}}$ in place of $\hat f$ and $\hat h^{(1)}$, 
provided $F^r$ and $|H^{(i)}|^2$ are 
$\sigma$-integrable relative to $\cG_0$ for $i=0,1$ for some $r>1$.  
Thus it remains to prove \eqref{p1} and \eqref{p4} with 
$^o\!f$ and $^o\!{h^{(1)}}$ 
in place of $\hat f$ and $\hat h^{(1)}$, respectively, 
under the condition that $F$ and $H^{(i)}$ are  
$\sigma$-integrable relative to $\cG_0$ for $i=0,1$, 
and \eqref{ph} holds.  
We show only \eqref{p4} under these conditions, because 
\eqref{p1} can be proven similarly. To this end  
 define $h^{(1)n}={\bf1}_{\frZ^n}(-n\vee h^{(1)}\wedge n)$
for integers $n\geq1$, where $(\frZ^n)_{n=1}^{\infty}$ is an increasing sequence 
of sets $\frZ^n\in \cZ_1$ 
such that $\bigcup_{n=1}^{\infty}\frZ^n=\frZ_1$ 
and $\nu_1(\frZ^n)<\infty$ for every $n\geq1$. 
Then for each $t\in[0,T]$ 
\begin{equation}                                                                                \label{eqn}
\E(\delta^{(1)n}_t|\cG_t)
=\int_0^t\int_{\frZ_1}{^o\!{h}^{(1)n}_s(\frz)}\,\tilde N_1(d\frz,ds)\quad\rm{(a.s.)}, 
\end{equation}
where $\delta^{(1)n}$ is defined as $\delta^{(1)}$ in \eqref{abc}, 
but with $h^{(1)n}$ in place of $h^{(1)}$. 
Note that 
$$
|^o\!{h}^{(1)n}|\leq |^o\!{h}^{(1)}|
\quad
\text{$P\otimes dt\otimes \nu_1$-almost every 
$(\omega,t,\frz)\in\Omega\times[0,T]\times\frZ_1$}, 
$$
and for $n\to$ we have 
${^o\!{h}^{(1)n}_s(\frz)}\to{^o\!{h}^{(1)}_s(\frz)}$ almost surely 
for every $(s,\frz)\in[0,T]\times\frZ_1$ such that 
${^o\!{h}^{(1)}_s(\frz)}\neq\infty$. Hence due to condition \eqref{ph}, 
by Lebesgue's theorem on dominated convergence 
we have 
$$
\int_0^T\int_{\frZ_1}|^o\!{h}^{(1)n}_s(\frz)-{^o\!{h}^{(1)}_s}(\frz)|^2\,\nu_1(d\frz)dt\to0
\,\,\rm{(a.s.)}\,\,\text{as $n\to\infty$}, 
$$
which implies 
\begin{equation}                                                                                     \label{oeqn}
\int_0^t\int_{\frZ_1}{^o\!{h}^{(1)n}_s(\frz)}\,\tilde N_1(d\frz,ds)
\to
\int_0^t\int_{\frZ_1}{^o\!{h}^{(1)}_s(\frz)}\,\tilde N_1(d\frz,ds)
\end{equation}
in probability, uniformly in $t\in[0,T]$. Using obvious properties of conditional 
expectations, by Davis inequality and Lebesgue's theorem 
on dominated convergence we get  
\begin{align*}
\lim_{n\to\infty}\E|\E(\delta^{(1)n}_t|\cG_t)-\E(\delta^{(1)}_t|\cG_t)|
&\leq \lim_{n\to\infty}\E|\delta^{(1)n}_t-\delta^{(1)}_t|                                            \\
&\leq 3\lim_{n\to\infty}
\E\left(\int_0^T\int_{\frZ_1}|h^{(1)n}_s(\frz)-h^{(1)}_s(\frz)|^2\,\,\nu_1(d\frz)ds\right)^{1/2}=0, 
\end{align*}
which by virtue of \eqref{eqn} and \eqref{oeqn} finishes 
the proof of the first equation in 
 \eqref{p4}. The second equation in \eqref{p4} can be obtained similarly. 
\end{proof}

\begin{remark}                                                                                                         \label{remark oah}
We have that almost surely 
$$
\int_0^t\int_{\frZ_i}{|^o\!h^{(i)}_s}(\frz)|^2\,\nu_i(d\frz)\,ds
\leq 
\int_0^t{^o\!(|h^{(i)}_s|^2_{L_2(\frZ_i)})}\,ds\quad\text{\rm{(a.s.)} for $i=0,1$},  
$$
for all $t\in[0,T]$. Thus 
\begin{equation}                                                                                                    \label{oah}
\int_0^T{^o\!(|h^{(i)}_t|^2_{L_2(\frZ_i)})}\,dt<\infty\quad\text{\rm{(a.s.)} for $i=0,1$}
\end{equation}
implies assumption \eqref{ph}.
\end{remark}

\begin{proof}
Let $i\in\{0,1\}$ be fixed and let $(A_n)_{n=1}^{\infty}$ be an increasing sequence of sets from $\frZ_i$ 
such that $\cup_{n=1}^{\infty}A_n=\frZ_i$ and $\nu_i(A_n)<\infty$ for every $n\geq1$. Set 
$$
h^{i,n}_t(\frz):=(-n)\vee ({\bf1}_{A_n}h_t^{(i)}(\frz))\wedge n.
$$
Then by Jensen's inequality for the optional projections we have 
$|{^oh^{i,n}_s(\frz)}|^2\leq {^o(|h^{i,n}_s(\frz)|^2)}$ for every 
$\frz\in\frZ_i$,  and 
by an application of Corollary \ref{corollary projection} we obtain 
\begin{align*}
\int_0^t\int_{\frZ_i}|{^oh^{i,n}_s(\frz)}|^2\,\nu_i(d\frz)ds
&\leq\int_0^t\int_{\frZ_i} {^o(|h^{i,n}_s(\frz)|^2)}\,\nu_i(d\frz)ds                                                            \\
&=\E\left(\int_0^t\int_{\frZ_i}|h^{i,n}_s(\frz)|^2\,\nu_i(d\frz)ds\Big|\cG_t\right)                                      \\
&=\int_0^t{^o(|h^{i,n}_s|^2_{L_2(\frZ_i)})}\,ds
\leq \int_0^t{^o(|h^{(i)}_s|^2_{L_2(\frZ_i)})}\,ds. 
\end{align*}
Letting here $n\to\infty$ 
and using the Monotone Convergence Theorem 
and the properties of extended optional projections on the left-hand side 
of the first inequality, we finish the proof of the remark. 
\end{proof}

Let  $\bP(\bR^{d})$ be the space of 
of probability measures on the Borel sets of $\bR^d$,   
equipped with the topology of weak convergence of measures. 
Let $C_b(\bR^d)$ denote the space of bounded continuous real functions on $\bR^d$, 
and as before, let $(\frZ,\cZ)$ be a separable measurable space.

\begin{lemma}                                                                   \label{lemma pmeasure}                                                                                 
Let $(\Omega, \cF,P)$ be a complete probability space 
equipped with a right-continuous 
filtration $(\cG_t)_{t\geq0}$, $\cG_t\subset\cF$ 
for $t\geq0$,  such that $\cG_0$ 
contains all $P$-zero sets of $\cF$. 
Let  $(X_t)_{t\geq0}$ be an $\bR^d$-valued 
$\cF\otimes\cB(\bR_+)$-measurable 
cadlag process. Then the following statements hold. 
\begin{enumerate}
\item[(i)] There is a $\bP(\bR^d)$-valued weakly 
cadlag  process $(P_t)_{t\geq0}$
such that for every bounded 
real-valued Borel function $\varphi$ on $\bR^d$ 
and for each $t\geq0$
\begin{equation}                                                                                   \label{pi}
P_t(\varphi)=\E(\varphi(X_t)|\cG_t)\,({\rm{a.s.}}).
\end{equation}
\item[(ii)]
Let $(P_t)_{t\geq0}$ be the measure-valued process from (i).  
Assume $f=f(\omega,t,\frz,x)$ is a 
$\mathcal O_{\cG}\otimes\cZ\otimes\cB(\bR^d)$-measurable 
real  function on 
$\Omega\times\bR_+\times\frZ\times\bR^d$. 
Define
$$
P_t(f(t,\frz)):=
\begin{cases} 
\int_{\bR^d}f(t,\frz,x)\,P_t(dx), 
\text{ for $(t,\omega,\frz)$, if }\int_{\bR^d}|f(t,\frz,x)|\,P_t(dx)<\infty\\
\infty \,\,\text{elsewhere.}
\end{cases}
$$
Then $P_t(f(t,\frz))$ is an $\mathcal O_{\cG}\otimes\cZ$-measurable 
(extended) function of $(\omega,t,\frz)$ such that 
\begin{equation}                                                                                 \label{exp}
\E(f(t,\frz,X_t)|\cG_t)=P_t(f(t,\frz))\,  \rm{ (a.s.)}
\quad
\text{for each $(t,\frz)\in\bR_+\times\frZ$}.  
\end{equation}
Moreover, $P_{t_0}(f(t_0,\frz_0))$ is finite (a.s.) 
for a pair $(t_0,\frz_0)\in \bR_+\times\frZ$ if $f(t_0,\frz_0,X_{t_0})$ 
is $\sigma$-integrable relative to $\cG_{t_0}$.
\end{enumerate}
\end{lemma}

\begin{proof}
Statement (i) is shown in \cite{Y}. Thus (ii) holds if $f=g(t,\frz)\varphi(x)$ 
for bounded $\mathcal O_{\cG}\otimes\cZ$-measurable functions $g$ 
on $\Omega\times\bR_+\times\frZ$ 
and bounded Borel functions $\varphi$ on $\bR^d$. 
Hence by a standard monotone class argument 
we get (ii) under the additional assumption that $f$ is bounded. 
In the general case, 
the set $A\subset\Omega\times\bR_+\times\frZ$ 
where 
$$
\int_{\bR^d}|f(t,\frz,x)|P_t(dx)=\infty
$$
is in $\mathcal{O}_{\cG}\otimes\cZ$. Consequently, 
$P_t(f(t,\frz))$ is $\cO_{\cG}\otimes\cZ$-measurable in 
$(\omega,t,\frz)$. We have 
$$
\E(|f(t,\frz, X_t)|\wedge n|\cG_t)=\int_{\bR^d}|f(t,\frz,x)|\wedge n\,P_t(dx) \quad\text{(a.s.)}
$$
for every integer $n\geq1$.  Letting here $n\to\infty$ we get 
\begin{equation}                                                               \label{veges}
\E(|f(t,\frz, X_t)||\cG_t)=\int_{\bR^d}|f(t,\frz,x)|\,P_t(dx)\quad\text{(a.s.)},   
\end{equation}
that implies \eqref{exp}.
If $f(t_0,\frz_0, X_{t_0})$ is $\sigma$-integrable 
relative to $\cG_{t_0}$, then 
there is an increasing sequence $(\Omega_n)_{n=1}^{\infty}$ 
such that $\Omega_n\in\cG_{t_0}$, 
$P(\cup_{n=1}^{\infty}\Omega_n)=1$, and 
$$
{\bf1}_{\Omega_n}\int_{\bR^d}f(t_0,\frz_0,x)\,P_{t_0}(dx)
=\E({\bf1}_{\Omega_n}f_n(t_0,\frz_0, X_{t_0})|\cG_{t_0})
$$
is almost surely finite for every $n\geq1$.
\end{proof}
\begin{corollary}                                                                       \label{corollary mp}
Let $(\Omega, \cF, P, (\cG_t)_{t\geq0})$ and $(X_t)_{t\in[0,T]}$ 
be a filtered probability space 
and a stochastic process, respectively, 
satisfying the conditions in Lemma \ref{lemma pmeasure}. 
Let $(\cF_t)_{t\geq0}$ be a filtration such that $\cG_t\subset\cF_t\subset\cF$ 
for $t\geq0$. Let $Q$ be a probability measure on $\cF$ such that 
$dQ=\gamma_TdP$ for a $\cF_T$-measurable positive random variable 
$\gamma_T$. 
Then the following statements hold. 
\begin{enumerate}
\item[(i)] There is an $\bM(\bR^d)$-valued weakly cadlag 
stochastic process $(\mu_t)_{t\in[0,T]}$ 
such that for every bounded real-valued Borel function 
$\varphi$ on $\bR^d$ and for every $t\in[0,T]$
\begin{equation}                                                                           \label{mu}
\mu_t(\varphi)=\E_Q(\gamma^{-1}_T\varphi(X_t)|\cG_t)
=\E_Q(\gamma^{-1}_t\varphi(X_t)|\cG_t)\,({\rm{a.s.}}).
\end{equation}
\item[(ii)]
Let $f=f(\omega,t,\frz,x)$ be a 
$\mathcal O_{\cG}\otimes\cZ\otimes\cB(\bR^d)$-measurable real  function on 
$\Omega\times[0,T]\times\frZ\times\bR^d$. Define
$$
\mu_t(f(t,\frz)):=\begin{cases} \int_{\bR^d}f(t,\frz,x)\,\mu_t(dx),  
\text{ for $(t,\omega,\frz)$, if } 
\int_{\bR^d}|f(t,\frz,x)|\,\mu_t(dx)<\infty\\
\infty \,\,\text{elsewhere.}\end{cases}
$$
Then $\mu_t(f(t,\frz))$ 
is a $\mathcal O_{\cG}\otimes\cZ$-measurable function 
such that for each $(t,\frz)$ we have 
\begin{equation}                                                                                        \label{muf}
\E_Q(\gamma^{-1}_Tf(t,\frz,X_t)|\cG_t)=
\E_Q(\gamma^{-1}_tf(t,\frz,X_t)|\cG_t)=\mu_t(f(t,\frz))\quad \text{(a.s.)}.  
\end{equation}
Moreover, $\mu_{t_0}(f(t_0,\frz_0))$ is finite (a.s.) 
for a $(t_0,\frz_0)\in[0,T]\times\frZ$ 
if $\gamma^{-1}_{t_0}f(t_0,\frz_0,X_{t_0})$ 
is $\sigma$-integrable (with respect to $Q$) relative to $\cG_{t_0}$, 
or equivalently, if $f(t_0,\frz_0,X_{t_0})$ is $\sigma$-integrable 
with respect to $P$ relative to $\cG_{t_0}$.
\end{enumerate}
\end{corollary}

\begin{proof}
Considering $(\cF_{t+})_{t\geq0}$ in place of $(\cF_{t})_{t\geq0}$ we may assume 
in the proof that $(\cF_{t})_{t\geq0}$ is right-continuous. 
By Doob's theorem there is a cadlag $\cF_{t}$-martingale, $(\gamma_t)_{t\in[0,T]}$,  
such that $\gamma_t=\E_P(\gamma_T|\cF_{t})$ ($P$-a.s) for each $t\in[0,T]$. 
Clearly, almost surely $\gamma_t>0$ for all $t\in[0,T]$ since  
$$
0=\E_P({\bf1}_{\gamma_t=0}\gamma_t)=\E_P({\bf1}_{\gamma_t=0}\gamma_T)
$$
implies $P(\gamma_t=0)=0$ for every $t\in[0,T]$.
Thus $(\gamma^{-1}_t)_{t\in[0,T]}$ is a cadlag  process, and it is an $\cF_t$-martingale 
under $Q$, because  
$$
\E_Q(\gamma^{-1}_T|\cF_t)=1/\E_P(\gamma_T|\cF_t)=\gamma^{-1}_t\quad\text{almost surely for $t\in[0,T]$}.
$$
Since, $\gamma=(\gamma)_{t\in[0,T]}$ 
is a (cadlag) $\cF_t$-martingale under $P$, the set $\{\gamma_{\tau}\}$ 
for $\cF_t$-stopping times $\tau\leq T$ is 
uniformly $P$-integrable, and hence one knows that $^o\!\gamma$, the $\cG_t$-optional 
projection of $\gamma$ under $P$, is a cadlag process. Due to $\gamma>0$, 
we have $^o\!\gamma>0$ (a.s.). 
Define $\mu_t:=(^o\!\gamma_t)^{-1}P_t$ 
for $t\in[0,T]$, where $(P_t)_{t\in[0,T]}$ is the $\bP(\bR^d)$-valued $\cG_t$-adapted cadlag 
process (in the topology of weak convergence of measures) by 
Lemma \ref{lemma pmeasure}. 
Hence, $(\mu_t)_{t\in[0,T]}$ is a $\cG_t$-adapted cadlag 
$\bM(\bR^d)$-valued process, and by \eqref{pi} for every bounded 
Borel function $\varphi$ on $\bR^d$ 
we have 
$$
\E_Q(\gamma^{-1}_T\varphi(X_t)|\cG_t)=\E_P(\varphi(X_t)|\cG_t)/\E_P(\gamma_T|\cG_t)
$$
$$
=\E_P(\varphi(X_t)|\cG_t)(^o\!\gamma_t)^{-1}=(^o\!\gamma_t)^{-1}P_t(\varphi)=\mu_t(\varphi)
\quad\text{(a.s) for each $t\in[0,T]$}.    
$$
On the other hand, by well-known properties of conditional expectations 
$$
\E_Q(\gamma^{-1}_T\varphi(X_t)|\cG_t)=
\E_Q(\E_Q(\gamma^{-1}_T\varphi(X_t)|\cF_t)|\cG_t)
$$
$$
=\E_Q(\varphi(X_t)\E_Q(\gamma^{-1}_T|\cF_t)|\cG_t)=
\E_Q(\gamma^{-1}_t\varphi(X_t))|\cG_t), 
$$ 
which completes the proof of (i). 
To prove (ii), note 
that the function $\mu_t(f(t,\frz))$ is $\cO_\cG\otimes\cZ$-measurable  
in $(\omega,t,\frz)$,
and by \eqref{exp} for each $(t,\frz)$ almost surely 
$$
\mu_t(f(t,\frz))=({^o\!\gamma_t})^{-1}P_t(f(t,\frz))
=\E(({^o\!\gamma_t})^{-1}f(t,\frz, X_t)|\cG_t)
$$
$$
=\E_Q(\gamma^{-1}_T({^o\!\gamma_t})^{-1}f(t,\frz, X_t)|\cG_t)/\E_Q(\gamma^{-1}_T|\cG_t)=
\E_Q(\gamma^{-1}_T({^o\!\gamma_t})^{-1}f(t,\frz, X_t)|\cG_t){^o\!\gamma_t}
$$
$$
=\E_Q(\gamma^{-1}_Tf(t,\frz, X_t)|\cG_t)=\E_Q(\gamma^{-1}_tf(t,\frz, X_t)|\cG_t), 
$$
where the last equation holds because 
$\gamma^{-1}$ is an $\cF_t$-martingale under $Q$. 
We finish the proof with the obvious observation that  
$\gamma^{-1}_{t_0}f(t_0,\frz_0,X_{t_0})$ is $\sigma$-integrable 
with respect to $Q$ relative to $\cG_{t_0}$
if $f(t_0,\frz_0,X_{t_0})$ is $\sigma$-integrable with respect to 
$P$ relative to $\cG_{t_0}$. 
 \end{proof}

\mysection{Proof of Theorem \ref{theorem Z1}}

Recall that by Assumption \ref{assumption Girsanov} 
the measure $Q$, defined by 
$dQ=\gamma_TdP$ is a probability measure, equivalent to $P$, 
and by Girsanov's theorem under $Q$ the process 
$(W_t,\tilde V_t)_{t\in[0,T]}$, where $(\tilde V_t)_{t\in[0,T]}$ 
is defined by \eqref{tilde Wiener}, is a $d_1+d'$-dimensional 
$\cF_t$-Wiener process. 
Moreover, under $Q$ the random measures 
$\tilde N_0$ and $\tilde N_1$ remain 
independent $\cF_t$-Poisson martingale measures, 
with characteristic measures $\nu_0$ 
and $\nu_1$, respectively. 
Clearly, $(\gamma_t)_{t\in[0,T]}$ is an 
$\cF_t$-martingale under $P$.  
By It\^o's formula 
\begin{equation}                                                                          \label{1.4.5.22}
d\gamma_t^{-1}=\gamma^{-1}_tB^l_t(X_t)\,d\tilde V^l_t,  
\end{equation}
the process $\gamma^{-1}=(\gamma^{-1}_t)_{t\in[0,T]}$ i
s an $\F_t$-local martingale under $Q$. 
Hence, taking into account $\E_Q\gamma_T^{-1}=1$, 
we get that $\gamma^{-1}$ is 
an $\cF_t$-martingale under $Q$. Thus the Bayes formula 
for bounded Borel functions $\varphi$ on $\bR^d$ gives  
\begin{equation}                                                            \label{Bayes}
\E(\varphi(X_t)|\cF^{Y}_t)
=\frac{\E_Q(\gamma_T^{-1}\varphi(X_t)|\cF^{Y}_t)}
{\E_Q(\gamma_T^{-1}|\cF^{Y}_t)}
=\frac{\E_Q(\gamma_t^{-1}\varphi(X_t)|\cF^{Y}_t)}
{\E_Q(\gamma_t^{-1}|\cF^{Y}_t)} 
\,\,({\rm{a.s}}.), 
\end{equation}
often referred as Kallianpur-Striebel formula in the literature. 
Using $\tilde{V}$ we can rewrite  system \eqref{system_1} 
in the form
\begin{align}
dX_t=& b(t,Z_t)\,dt + \sigma(t,Z_t)\,dW_t + \rho(t,Z_t)\,dV_t                 \nonumber\\
 & +\int_{\frZ_0}\eta(t,Z_{t-},\frz)\,\tilde N_0(d\frz,dt) 
 + \int_{\frZ_1}\xi(t,Z_{t-},\frz)\, \tilde N_1(d\frz,dt) \nonumber\\
    dY_t =& d\tilde V_t + \int_{\frZ_1}\frz\,\tilde{N}_1(dt,d\frz),                  \label{SDE2} 
\end{align}
which shows, in particular, that $(Y_t)_{t\in[0,T]}$ is a L\'evy process 
under $Q$, 
and hence it is well-known that 
the filtration $(\cF_t^{Y})_{t\in[0,T]}$ 
is right-continuous. Thus we can apply 
Lemma \ref{lemma pmeasure} and 
Corollary \ref{corollary mp} with the 
unobservable process $(X_t)_{t\in[0,T]}$ 
and the filtration $(\cG_t)_{t\in[0,T]}=(\cF_t^Y)_{t\in[0,T]}$ 
to have a $\bP$-valued and $\bM$-valued 
weakly cadlag $\cF^{Y}_t$-adapted processes $P_t(dx)$ and $\mu_t(dx)$, 
respectively,  such that for every bounded 
Borel function $\varphi$ on 
$\bR^d$ for each $t\in[0,T]$ we have 
$$
P_t(\varphi)=\E(\varphi(X_t)|\cF^{Y}_t), 
\quad 
\mu_t(\varphi)=\E_Q(\gamma^{-1}_t\varphi(X_t)|\cF^{Y}_t)\,\,({\rm{a.s.}}), 
$$
and by \eqref{Bayes} it follows that almost surely 
$P_t=\mu_t/\mu_t(\bf 1)$ for all $t\in[0,T]$.
To get an equation for $d\mu_t(\varphi)$ 
for sufficiently smooth functions we 
calculate first the stochastic differential 
$d(\gamma^{-1}_t\varphi(X_t))$.  
\begin{proposition}                                                                              \label{proposition SD}
Let $\varphi\in C^2_b(\bR^d)$. 
Then for the stochastic differential of $\gamma^{-1}_t\varphi(X_t)$ we have
\begin{equation}                                                                                     
 \begin{split}
  d\big(\gamma^{-1}_t\varphi(X_t)\big)
 =& \gamma^{-1}_t\cL_t\varphi(X_t)\,dt 
  + \gamma^{-1}_t\cM_t^l\varphi(X_t)\,d\tilde V^l_t 
  + \gamma^{-1}_t \sigma^{ik}_t(X_t)D_i\varphi(X_t)\,dW^{k}_t                                              \\
&+\gamma^{-1}_t\int_{\frZ_0}I_t^{\eta}\varphi(X_{t-})\,\Ntn(d\frz,dt)
+\gamma^{-1}_t\int_{\frZ_1}I_t^{\xi}\varphi(X_{t-})\,\Nte(d\frz,dt) \\
&+ \gamma^{-1}_t\int_{\frZ_0}J_t^{\eta}\vp(X_t)\,\nu_0(d\frz)\,dt 
 +\gamma^{-1}_t\int_{\frZ_1}J_t^{\xi}\varphi(X_t)\,\nu_1(d\frz)\,dt.
    \end{split}                                                                                  \label{SD}
\end{equation}
\end{proposition}

\begin{proof}
By It\^o's formula, see for example in \cite{A} or \cite{IW}, 
for $\varphi\in C_b^{2}(\bR^d)$ we have 
\begin{align*}
d\varphi(X_t)=
& \left(\cL_t\varphi(X_t)-\rho^{il}_tB^{l}_t(X_t)D_i\varphi(X_t)\right)\,dt              \\
&+ \sigma_t^{ik}(X_t)D_i\varphi(X_t)\,dW^k_t 
+ \rho_t^{il}(X_t)D_i\varphi(X_t)\,d\tilde V^l_t          \\
&+ \int_{\frZ_0} I_t^{\eta}\varphi(X_{t-})\,\tilde N_0(dt,d\frz)
+ \int_{\frZ_1} I_t^{\xi}\varphi(X_{t-})\,\tilde N_1(dt,d\frz)\\
&+\int_{\frZ_0} J_t^{\eta}\varphi(X_{t-})\,\nu_0(d\frz)dt + \int_{\frZ_1} J_t^{\xi}\varphi(X_{t-})\,\nu_1(d\frz)dt, 
\end{align*}
where we use the notations introduced before the formulation of Theorem \ref{theorem Z1}. 
Hence using \eqref{1.4.5.22} and 
the stochastic differential rule for products, 
$$
d(\gamma_t^{-1}\varphi(X_t))=\gamma^{-1}_{t}d\varphi(X_t)+\varphi(X_{t-})\,d\gamma_t^{-1}
+
d\gamma^{-1}_{t}d\varphi(X_t),
$$
where
$$
d\gamma^{-1}_{t}d\varphi(X_t)=\gamma^{-1}_t\rho_t^{il}B_t^l(X_t)D_i\varphi(X_t)\,dt, 
$$
we obtain \eqref{SD}.
\end{proof}
To calculate the conditional expectation (under $Q$) of the terms 
in the equation for $\gamma_t^{-1}\varphi(X_t)$, given $\cF_t^{Y}$, 
we describe below the structure of $\cF_t^{Y}$. 
For each $t\geq0$ we denote by $\cF^{\tilde N}_t$ the $P$-completion of the  
$\sigma$-algebra generated by the random variables 
$N_1((0,s]\times\Gamma)$ for $s\in(0,t]$ and $\Gamma\in\cZ_1$ 
such that $ \nu_1(\Gamma)<\infty$. 

\begin{lemma}                                                            \label{lemma sa}
For every $t\in[0,T]$ we have 
\begin{equation*}
    \F_t^Y = \F_0^Y \vee {\F_t^{\tilde{V}}} \vee \F_t^{\Nte}, 
\end{equation*}
where $\F_0^Y \vee {\F_t^{\tilde{V}}} \vee \F_t^{\Nte}$ 
denotes the $P$-completion of the 
smallest $\sigma$-algebra containing  $\F_0^Y$, 
$\F_t^{\tilde{V}}$  and $\F_t^{\tilde N_1}$. 
\end{lemma}

\begin{proof}
From \eqref{SDE2} it immediately follows that
\begin{equation*}
    \F_t^Y \subseteq \F_0^Y \vee {\F_t^{\tilde{V}}} \vee \F_t^{\Nte}.
\end{equation*}
To prove the reversed inclusion, we claim 
\begin{equation}                                                                                \label{NN1}
N^Y((0,t]\times A)=N_1((0,t]\times A)
\quad
\text{almost surely for all $t\in [0,T]$}
\end{equation}
for every $A\in\cZ_1$, where $N^Y$ is the measure of jumps for the process $Y$. Clearly, 
$N^Y(d\frz,dt)=N^M(d\frz,dt)$, where $N^M$ is the measure of jumps  
for the process
$$
M_t=\int_0^t\int_{\frZ_1}\frz\,\tilde N_1(d\frz,dt), \quad t\geq0, 
$$
i.e., 
$$
N^Y((0,t]\times A)=\sum_{0<s\leq t}{\bf1}_{A}(\Delta Y_s)
=\sum_{0<s\leq t}{\bf1}_{A}(\Delta M_s)
\quad A\in\cZ_1. 
$$
To show \eqref{NN1} let $A=A_0$ be a set from $\cZ_1$ such that $\nu_1(A_0)<\infty$. 
Then 
$$
M^{A_0}_t:=\int_0^t\int_{A_0}\frz\,\tilde N_1(d\frz,ds)=
\sum_{0<s\leq t}p_s{\bf1}_{A_0}(p_s)-t\int_{A_0}\frz\,\nu_1(d\frz), 
$$
where $(p_t)_{t\in [0,T]}$ is the Poisson point process associated with $N_1$. 
Hence for $N^0(d\frz,dt)$, the measure 
of jumps of the process $M^{A_0}$, we have that almost surely 
\begin{equation}                                                                                                    \label{01}
N^0((0,t]\times A_0)=N_1((0,t]\times A_0)
\quad
\text{for all $t\in[0,T]$}. 
\end{equation}
It is not difficult to see that $N^0((0,t]\times A_0)=N^M((0,t]\times A_0)$. 
Hence \eqref{NN1} for $A=A_0$ follows. 

Since $\nu_1$ is $\sigma$-finite, for an arbitrary $B\in\cZ_1$ 
there is a sequence $(B_n)_{n=1}^{\infty}$ 
of disjoint sets $B_n\in\cZ_1$ such that $B=\bigcup_{n=1}^{\infty}B_n$ 
and $\nu_1(B_n)<\infty$ for 
each $n\geq1$. Thus for each integer $n\geq 1$ we have \eqref{NN1} 
with $B_n$ in place of $A$, and 
summing this up over $n\geq1$ and using the $\sigma$-additivity of $N^Y$ 
and $N_1$ we obtain 
\eqref{NN1} with $B$ in place of $A$. Noting that $\cF^Y_t$ 
contains the $\sigma$-algebra 
generated by $N^Y((0,s]\times B)$ for each $s\leq t$ and $B\in\cZ$, 
we see that $\cF^Y_t\supset \cF_t^{\tilde N_1}$. 
Clearly, $\cF_t^Y\supset \cF^Y_0$,  and taking into account  
$$
Y_t-Y_0-\int_0^t\int_{\frZ_1}\frz\,N_1(d\frz,ds)=\tilde V_t, 
\quad 
\text{for $t\in[0,T]$},
$$
we get $\cF^Y_t\supset \cF_t^{\tilde V}$. Consequently, 
\begin{equation*}
    \F_t^Y \supseteq \F_0^Y \vee {\F_t^{\bar{V}}} \vee \F_t^{\tilde N}, 
\end{equation*}
that completes the proof. 
\end{proof}

The above lemma is an essential tool in obtaining the filtering equations. 
A similar lemma in a more general setting 
in some directions is presented in  \cite{QD2015} and \cite{Q2019} 
to obtain the filtering equations 
for the model presented in these papers.  It seems to us, however, that this 
lemma, Lemma 3.2 in  \cite{QD2015}, used as well in \cite[p.4]{Q2019}, 
may not hold under the general conditions 
formulated in these papers, since it is not true in simple cases of 
vanishing coefficients in the equation forof the observation noises. It is worth noticing that 
when instead 
of the integrand $\frz$ a stochastic integrand depending on $Z_t=(X_t,Y_t)$ is 
considered in the observation process $Y$, the integral of such a term against 
a Poisson random measure may fail to be a L\'evy process. In particular, 
it may not have independent increments, which is a crucial property 
for the filtration generated by the observation.

Now we are going to get an equation for $\mu(\varphi)$ by noting that by 
Proposition \ref{proposition SD} we have 
\begin{equation}                                                                                                 \label{Ito}
\gamma^{-1}_t\varphi(X_t)=\varphi(X_0)
+\alpha_t+\alpha^{0}_t+\alpha^{1}_t+\beta_t^{0}+\beta_t^{1}+\delta_t^0+\delta^{1}_t,\quad t\in[0,T], 
\end{equation}
where 
$$
\alpha_t:=\int_0^t\gamma^{-1}_s\cL_s\varphi(X_s)\,ds, 
$$
$$
\alpha^0_t:=\int_0^t\int_{\frZ_0}\gamma^{-1}_sJ_s^{\eta}\vp(X_s)\,\nu_0(d\frz)ds, 
\quad
\alpha^1_t:=\int_0^t\int_{\frZ_1}\gamma^{-1}_sJ_s^{\xi}\varphi(X_s)\,\nu_1(d\frz)ds, 
$$
$$
\beta^0_t:=\int_0^t\gamma^{-1}_s \sigma^i_s(X_s)D_i\varphi(X_s)\,dW_s, 
\quad
\beta^1_t:=\int_0^t\gamma^{-1}_s\cM_s^l\varphi(X_s)\,d\tilde V^l_s, 
$$
$$
\delta^0_t:=\int_0^t\int_{\frZ_0}\gamma^{-1}_sI_s^{\eta}\varphi(X_{s-})\,\tilde N_0(d\frz,ds),\quad
\delta^1_t:=\int_0^t\int_{\frZ_1}\gamma^{-1}_sI_s^{\xi}\varphi(X_{t-})\,\tilde N_1(d\frz,ds), 
$$
for $\varphi\in C^2_b(\bR^d)$. 
We want to take the conditional expectation of both sides 
of equation \eqref{Ito} for each $t\in[0,T]$, under $Q$, given $\cF^Y_t$. 
In order to apply Corollary  \ref{corollary projection},   
 we should verify that the random variables 
$$
G:=\int_0^T\gamma^{-1}_s|\cL_t\varphi(X_s)|\,ds, 
$$
$$
G^{(0)}:=\int_0^T\int_{\frZ_0}
\gamma^{-1}_s|J_s^{\eta}\varphi(X_s)|\,\nu_0(d\frz)\,ds,  
\quad 
G^{(1)}:=\int_0^T\int_{\frZ_1}
\gamma^{-1}_s|J_s^{\xi}\varphi(X_s)|\,\nu_1(d\frz)\,ds,
$$
$$
F^{(0)}:=
\Big(\int_0^T\gamma^{-2}_s|\sigma^i_s(X_s)D_i\varphi(X_s)|^2\,ds
\Big)^{1/2}, 
\quad
F^{(1)}:=
\Big(\int_0^T\gamma^{-2}_s\sum_{l}|\cM_s^l\varphi(X_s)|^2\,ds
\Big)^{1/2}, 
$$
$$
H^{(0)}:=
\Big(
\int_0^T\int_{\frZ_0}\gamma^{-2}_s|I_s^{\eta}\varphi(X_{s-})|^2\,\nu_0(d\frz)ds
\Big)^{1/2}, 
$$
$$
H^{(1)}:=
\Big(
\int_0^T\int_{\frZ_1}\gamma^{-2}_s|I_s^{\xi}\varphi(X_{s-})|^2\,\nu_1(d\frz)ds
\Big)^{1/2}
$$
are $\sigma$-integrable with respect to $Q$ relative to $\cF_0^Y$, 
and that \eqref{ph} holds for $^Q\!f^{(i)}$ in place of $^o\!f$, 
and for $^Q\!h^{(i)}$ in place of $^o\!h^{(i)}$, where $^Q\!f^{(0)}$, $^Q\!f^{(1)}$,  
$^Q\!h^{(0)}$  and 
$^Q\!h^{(1)}$ are the $\cF^Y_t$-optional projection under $Q$ 
of 
$$
f^{k(0)}:=(\gamma^{-1}_s\sigma^{ik}_s(X_s)D_i\varphi(X_s))_{s\in[0,T]}, 
\quad
f^{l(1)}:=(\gamma^{-1}_s\cM_s^l\varphi(X_s))_{s\in[0,T]}, 
$$
$$
h^{(0)}:=(\gamma^{-1}_sI_s^{\eta}\varphi(X_{s-}))_{s\in[0,T]}
\quad
\text{and}
\quad
h^{(1)}:=(\gamma^{-1}_sI_s^{\xi}\varphi(X_{s-}))_{s\in[0,T]}, 
$$
respectively for each fixed $k=1,2,...,d_1$ and $l=1,...,d'$.  
For a fixed integer $n\geq1$ let $\Omega_n=\{\omega\in\Omega:|Y_0|\leq n\}$.  
Then due to Assumption \ref{assumption SDE},  
the martingale property of $(\gamma_t)_{t\in[0,T]}$  
and \eqref{bound_Z} we have 
$$
\E_Q({\bf1}_{\Omega_n}G)
\leq N\E\Big(\gamma_T\int_0^T\gamma^{-1}_t
(K_0+ K_1 {\bf1}_{\Omega_n}|Z_t| +K_2{\bf1}_{\Omega_n}|Z_t|^2)\,dt\Big)
$$
$$
=N\int_0^T\E(\gamma_T\gamma^{-1}_t
(K_0+ K_1 {\bf1}_{\Omega_n}|Z_t| +K_2{\bf1}_{\Omega_n}|Z_t|^2)\,dt
$$
$$
=N\int_0^T\E(K_0+K_1 {\bf1}_{\Omega_n}|Z_t| +K_2{\bf1}_{\Omega_n}|Z_t|^2)\,dt
$$
$$
\leq N'(K_0+K_1 \E|X_0| + K_1\E({\bf1}_{\Omega_n}|Y_0|)+  
K_2\E|X_0|^2+K_2\E({\bf1}_{\Omega_n}|Y_0|^2))<\infty
$$
with  constants $N$ and $N'$, which shows that $G$ 
is $\sigma$-integrable with respect to 
$Q$ relative to $\cF^Y_0$. 
Similarly, using the estimate
$$
|J^\eta\varphi(X_t)|\leq \sup_{x\in\bR^d}
|D_{ij}\varphi(x)||\eta^i_t(X_t)||\eta^j_t(X_t)|,
$$
we get 
$$
\E_Q({\bf1}_{\Omega_n}G^{(0)})
=\int_0^T\E\int_{\frZ_0}{\bf1}_{\Omega_n}
|J_s^{\eta}\varphi(X_s)|\,\nu_0(d\frz)ds
$$
$$
\leq N\int_0^T\E\int_{\frZ_0}{\bf1}_{\Omega_n}
|\eta(s,Z_s,\frz)|^2\,\nu_0(d\frz)\,ds
\leq N'\int_0^T\E(K_0+K_2{\bf1}_{\Omega_n}|Z_s|^2)\,ds<\infty
$$
with  constants $N$ and $N'$. 
In the same way we get 
$\E_Q({\bf1}_{\Omega_n}G^{(1)})<\infty$. 
To prove that $F^{(i)}$ and $H^{(i)}$ 
are $\sigma$-integrable (with respect to $Q$) relative to $F_0^{Y}$, 
we claim first that 
\begin{equation}                                                                       \label{A}
A_n:=\E_Q{\bf1}_{\Omega_n}\sup_{t\leq T}\gamma_t^{-1}<\infty
\quad\text{for every integer $n\geq1$}.
\end{equation}
To prove this we repeat a method used in proof of Theorem
\ref{theorem Girsanov}. From \eqref{1.4.5.22} by using the Davis inequality 
and then Young's inequality we get  
$$
\E_Q{\bf1}_{\Omega_n}\sup_{t\in[0,T]}\gamma^{-1}_{t\wedge\tau_k}
\leq
1+3\E\left(
\int_0^{T\wedge\tau_k}{\bf1}_{\Omega_n}\gamma_t^{-2}|B(t,Z_t)|^2\,dt
\right)^{1/2}
$$
$$
\leq1+\tfrac{1}{2}
\E_Q{\bf1}_{\Omega_n}\sup_{t\in[0,T]}\gamma^{-1}_{t\wedge\tau_k}
+
5\E\int_0^{T}{\bf1}_{\Omega_n}\gamma_t^{-1}|B(t,Z_t)|^2\,dt
$$
for stopping times 
$$
\tau_k=\inf\{t\in[0,T]: \gamma^{-1}_t\geq k\},
\quad\text{for integers $k\geq1$}. 
$$
Rearranging this inequality and then letting $k\to\infty$ 
by Fatou's lemma we obtain 
$$
\E_Q{\bf1}_{\Omega_n}\sup_{t\in[0,T]}\gamma^{-1}_{t}
\leq 2+10\int_0^T\E_Q{\bf1}_{\Omega_n}\gamma_t^{-1}|B(t,Z_t)|^2\,dt. 
$$
Hence we get \eqref{A} by noticing 
that using the martingale property of $\gamma$,  
the estimate in \eqref{bound_Z} and $K_2\E |X_0|^2<\infty$, 
for every $t\in[0,T]$ we have 
$$
\E_Q{\bf1}_{\Omega_n}\gamma_t^{-1}|B(t,Z_t)|^2
=\E{\bf1}_{\Omega_n}\gamma_T\gamma_t^{-1}|B(t,Z_t)|^2
=\E{\bf1}_{\Omega_n}|B(t,Z_t)|^2
$$
$$
\leq K_0+K_2\E{\bf1}_{\Omega_n}|Z_t|^2
\leq K_0+K_2N\E(1+|X_0|^2+{\bf1}_{\Omega_n}|Y_0|^2)<\infty. 
$$
Consequently, 
$$
\E_Q({\bf1}_{\Omega_n}F^{(0)})
\leq 
\E_Q
\left(
{\bf1}_{\Omega_n}\sup_{s\leq T}\gamma^{-1/2}_s
\Big(
\int_0^T{\bf1}_{\Omega_n}
\gamma^{-1}_s|\sigma^i_s(X_s)D_i\varphi(X_s)|^2\,ds
\Big)^{1/2}
\right)
\leq A_n+B_n, 
$$
with $A_n<\infty$, and  
$$
B_n:=\E_Q\int_0^T{\bf1}_{\Omega_n}
\gamma^{-1}_s|\sigma^i_s(X_s)D_i\varphi(X_s)|^2\,ds
=\int_0^T\E({\bf1}_{\Omega_n}
\gamma_T\gamma^{-1}_s|\sigma^i_s(X_s)D_i\varphi(X_s)|^2)\,ds
$$
$$
=\int_0^T\E|{\bf1}_{\Omega_n}\sigma^i_s(X_s)D_i\varphi(X_s)|^2\,ds
\leq N\int_0^T\E(K_0+K_2{\bf1}_{\Omega_n}|Z_s|^2)\,ds<\infty.
$$
We get $\E_Q({\bf1}_{\Omega_n}F^{(1)})<\infty$ in the same way. 
Similarly, 
$\E_Q({\bf1}_{\Omega_n}H^{(0)})\leq A_{n}+C_{n}$, 
with $A_{n}$ given in \eqref{A} and 
$$
C_{n}:=\E_Q
\int_0^T\gamma^{-1}_s\int_{\frZ_0}{\bf1}_{\Omega_n}|
I_s^{\eta}\varphi(X_s)|^2\,\nu_0(d\frz)ds
=\int_0^T\E\int_{\frZ_0}{\bf1}_{\Omega_n}
|I_s^{\eta}\varphi(X_s)|^2\,\nu_0(d\frz)ds
$$
$$
\leq N\int_0^T\E\int_{\frZ_0}{\bf1}_{\Omega_n}
|\eta(s,Z_s,\frz)|^2\,\nu_0(d\frz)ds
\leq N'\int_0^T\E(K_0+K_2{\bf1}_{\Omega_n}|Z_s|^2)\,ds<\infty
$$
with constants $N$ and $N'$, where we use that by Taylor's formula we have 
$$
|I_s^{\eta}\varphi(X_s)|\leq \sup_{x\in\bR^d}|D_i\varphi(x)||\eta^i_s(X_s)|.
$$
In the same way we have $\E_Q({\bf1}_{\Omega_n}H^{(1)})<\infty$. 
For processes $h=(h_t)_{t\in[0,T]}$ let 
$^Q\!h$ and  ${^{o}}\!h$ denote the $\cF^Y_t$-optional projections 
of $h$ under $Q$ and under $P$, 
respectively. Then 
using the formula $^Q\!h={^{o}\!(\gamma h)}/^{o}\!\gamma$,  
well-known properties of optional 
projections and Remark \ref{remark oah} we have 
$$
|^Q\!h^{0}|^2_{L_2(\frZ_0)}
=\frac{|^{o}\!(I^{\eta}\varphi(X))|^2_{L_2(\frZ_0)}}{(^{o}\!\gamma)^2}
\leq \frac{^{o}\!\big(|I^{\eta}\varphi(X))|^2_{L_2(\frZ_0)}\big)}{(^{o}\!\gamma)^2} 
$$
$$
\leq N\frac{^{o}\!(K_0+K_2|Z|^2)}{(^{o}\!\gamma)^2}
=N\frac{K_0}{(^{o}\!\gamma)^2}
+NK_2\frac{^{o}\!(|X|^2)}{(^{o}\!\gamma)^2}+NK_2\frac{|Y|^2}{(^{o}\!\gamma)^2}
$$
with a constant $N$. Remember that since $\gamma=(\gamma)_{t\in[0,T]}$ 
is a (cadlag) $\cF_t$-martingale under $P$, the set $\{\gamma_{\tau}\}$ 
for $\cF_t$-stopping times $\tau\leq T$ is 
uniformly $P$-integrable and hence due to the right-continuity 
of $(\cF_t^Y)_{t\in [0,T]}$, the optional projection 
$^{o}\!\gamma$ 
is a cadlag process. Moreover, due to $\gamma>0$, 
we have $^{o}\!\gamma>0$ (a.s.). Since by \eqref{bound_Z}
$$
K_2\E(\sup_{t\leq T}{\bf1}_{\Omega_n}|X_t|^2) <\infty
\quad
\text{for every $n\geq1$}, 
$$
(and $(\cF_t^Y)_{t\in [0,T]}$ is right-continuous), the process 
$K_2{^{o}\!(|X|^2)}$ is a cadlag process. Consequently, 
$K_0/|^{o}\!\gamma|^2$, ${K_2}^{o}\!(|X|^2)/|^{o}\!\gamma|^2$ 
and 
$|Y|^2/|^{o}\!\gamma|^2$ are cadlag processes. Hence 
$$
\int_0^T \frac{1}{(^{o}\!\gamma)^2_s}\,ds
+K_2\int_0^T\frac{^{o}\!(|X|^2)_s}{(^{o}\!\gamma)^2_s}\,ds
+K_2\int_0^T\frac{|Y_s|^2}{(^{o}\!\gamma)_s^2}\,ds<\infty
\quad(\rm{a.s.}),   
$$
that proves 
$$
\int_0^T\int_{\frZ_i}|^Q\!h^{i}_s|^2\,\nu_i(d\frz)\,ds<\infty\quad(\rm{a.s.})
$$
for $i=0$, and we get this for $i=1$ in the same way. 
By the same argument we have 
$$
\int_0^T|^Q\!f^{(i)}_s|^2\,\,ds<\infty\quad(\rm{a.s.})
\quad
\text{for $i=0,1$}.
$$
Thus we can apply Corollary \ref{corollary projection} 
to the processes $\alpha$, $\alpha^i$, $\beta^i$ and $\delta^i$ (i=0,1), and 
then use Corollary \ref{corollary mp}, to get 
$$
\E_Q(\alpha_t|\cF_t^{Y})=\int_0^t\mu_s(\cL_s\varphi)\,ds, 
$$
$$
\E_Q(\alpha^0_t|\cF_t^{Y})
=\int_0^t\int_{\frZ_0}\mu_s(J_s^{\eta}\varphi)\,\nu_0(d\frz)ds,  
\quad
\E_Q(\alpha^1_t|\cF_t^{Y})
=\int_0^t\int_{\frZ_1}\mu_s(J_s^{\xi}\varphi)\,\nu_1(d\frz)ds, 
$$
$$
\E_Q(\beta^0_t|\cF_t^{Y})=0, 
\quad
\E_Q(\beta^1_t|\cF_t^{Y})=\int_0^t\mu_s(\cM^l_s\varphi)\,d\tilde V^l_s, 
$$
$$
\E_Q(\delta^0_t|\cF_t^{Y})=0, \quad
\E_Q(\delta^1_t|\cF_t^{Y})
=\int_0^t\int_{\frZ_1}\mu_s(I_s^{\xi}\varphi)\,\tilde{N}_1(d\frz,ds)
$$
for $t\in[0,T]$ and $\varphi\in C^2_b(\bR^d)$ almost surely, where $(\mu_t)_{t\in[0,T]}$ 
is an $\bM(\bR^d)$-valued $\cF^Y_t$-adapted weakly cadlag process such that 
$$
\mu_t(\varphi):=\int_{\bR^d}\varphi(x)\,\mu_t(dx)=\E_Q(\gamma^{-1}_t\varphi(X_t)|\cF^Y_t)\quad(\rm{a.s.})
\quad\text{for each $t\in[0,T],$}
$$
for every bounded Borel function $\varphi$ on $\bR^d$. 
Using Lemma \ref{lemma p1} 
with random variables $X:=\varphi(X_0)$, $Y:=1$ and $\sigma$-algebras 
$\cG_1:=\cF^Y_0$, $\cG:=\cF_0$ and 
$\cG_2:=\cF^{\tilde V}_t\vee \cF^{\tilde{N}_1}_t$ we get 
$$
\E_Q(\varphi(X_0)|\cF^Y_t)=\E_Q(\varphi(X_0)|\cF^Y_0)
=\mu_0(\varphi)\quad(\rm{a.s.}). 
$$
Consequently, taking the conditional expectation of both sides 
of equation \eqref{Ito} under $Q$ given $\cF^Y_t$,  
we see that equation \eqref{eqZ1} holds for each $t\in[0,T]$ 
and $\varphi\in C^2_b(\bR^d)$ almost surely, 
that implies that for each $\varphi\in C^2_b(\bR^d)$ 
equation \eqref{eqZ1} holds almost surely for all $t\in[0,T]$, 
since we have cadlag processes in both sides of equation 
\eqref{eqZ1} for each $\varphi\in C^2_b(\bR^d)$. 
To prove \eqref{eqZ2} first notice that for $\varphi:=\bf1$ 
equation \eqref{eqZ1} gives 
$$
d\mu_t({\bf1})=\mu_t(B^k_t)\,d\tilde V^k_t, \quad \mu_0({\bf1})=1.  
$$
Since $\mu_t({\bf1})=(^o\!\gamma_t)^{-1}P_t({\bf1})
=(^o\!\gamma_t)^{-1}$, $t\in[0,T]$, is a continuous process such that 
$\mu_t({\bf1})=\E_Q(\gamma^{-1}_t|\cF^Y_t)$ (a.s.) for each $t\in[0,T]$, 
it is the $\cF^Y_t$-optional projection under $Q$ 
of the positive process $(\gamma^{-1}_t)_{t\in[0,T]}$. 
Hence 
$\lambda_t:=\mu_t({\bf1})$, $t\in[0,T]$, 
is a positive process, and by It\^o's formula 
$$
d\lambda^{-1}_t=-\lambda^{-2}_t\mu_t(B^k_t)\,d\tilde V^k_t+\lambda^{-3}_t\sum_{k}\mu^2_t(B^k_t)\,dt. 
$$
By It\^o's formula for the product $P_t(\varphi)=\lambda^{-1}_t\mu_t(\varphi)$ we have 
$$
dP_t(\varphi)=P_t(\cL_t\varphi)\,dt+P_t(\cM^k_t\varphi)\,d\tilde V^k_t+\int_{\frZ_0}P_t(J^{\eta}_t\varphi)\,\nu_0(d\frz)dt
+\int_{\frZ_1}P_t(J^{\xi}_t\varphi)\,\nu_1(d\frz)dt
$$
$$
+\int_{\frZ_1}P_t(I^{\xi}_t\varphi)\,\tilde{N}_1(d\frz,dt)
+\lambda^{-3}_t\mu_t(\varphi)\sum_k\mu^2_t(B^k_t)\,dt
$$
$$
-\mu_t(\varphi)\lambda^{-2}_t\mu_t(B^k_t)\,d\tilde V^k_t-\lambda^{-2}_t\mu_t(B^k_t)\mu_t(\cM^k_t\varphi)\,dt
$$
Hence noting that 
$$
\lambda^{-3}_t\mu_t(\varphi)\sum_k\mu^2_t(B^k_t)=P_t(\varphi)\sum_kP^2_t(B_t^k), 
\quad
\mu_t(\varphi)\lambda^{-2}_t\mu_t(B^k_t)=P_t(\varphi)P_t(B^k_t)
$$
$$
\lambda^{-2}_t\mu_t(B^k_t)\mu_t(\cM^k_t\varphi)=P_t(B^k_t)P_t(\cM^k_t\varphi), 
$$
we obtain 
$$
dP_t(\varphi)=P_t(\cL_t\varphi)\,dt+\big(P_t(\cM^k\varphi)-P_t(\varphi)P_t(B^k_t)\big)\,d\tilde V^k_t
-\big(P_t(\cM^k\varphi)-P_t(\varphi)P_t(B^k_t)\big)P_t(B^k_t)\,dt
$$
$$
+\int_{\frZ_0}P_t(J^{\eta}_t\varphi)\,\nu_0(d\frz)dt
+\int_{\frZ_1}P_t(J^{\xi}_t\varphi)\,\nu_1(d\frz)dt
+\int_{\frZ_1}P_t(I^{\xi}_t\varphi)\,\tilde{N}_1(d\frz,dt).  
$$
Since clearly, 
$$
\big(P_t(\cM^k\varphi)-P_t(\varphi)P_t(B^k_t)\big)\,d\tilde V^k_t
-\big(P_t(\cM^k\varphi)-P_t(\varphi)P_t(B^k_t)\big)P_t(B^k_t)\,dt
$$
$$
=\big(P_t(\cM^k\varphi)-P_t(\varphi)P_t(B^k_t)\big)\,d\bar V^k_t
$$ 
with the process $(\bar V_t)_{t\in[0,T]}$, given by 
$d\bar V_t=d\tilde V_t-P_t(B_t)\,dt$,  $\bar V_0=0$, 
this gives equation \eqref{eqZ2},  and finishes the proof 
of Theorem \ref{theorem Z1}.


\begin{thebibliography}{mm}


\bibitem{A} D. Applebaum, {L\'{e}vy Processes and Stochastic Calculus}, 
Cambridge University Press, 2001.

\bibitem{A} D. Applebaum and S. Blackwood, 
{The Kalman-Bucy filter for integrable L\'evy 
processes with infinite second moment}, J. Appl. Prob. 52 (2015), 636-648.

\bibitem{B2014} S. Blackwood, {L\'evy Processes and Filtering
Theory}, Dissertation, The University of Sheffield, 2014.


\bibitem{CCC2014} T. Cass, M. Clark and D. Crisan, {The filtering equations revisited}, Stochastic Analysis and Applications, Springer, Cham, 2014. 129-162. (2014)

\bibitem{C2006} C. Ceci, {Risk minimizing hedging for a partially observed high frequency data model}, Stochastics 78, 13-31 (2006).

\bibitem{CC2012} C. Ceci and K. Colaneri, {Nonlinear filtering for jump diffusion observations},  Advances in Applied Probability 44-3, 678-701 (2012).

\bibitem{CC2014} C. Ceci and K. Colaneri, {The Zakai equation of nonlinear filtering for jump-diffusion observations: existence and uniqueness}, 
Applied Mathematics \& Optimization 69, 47-82 (2014). 

\bibitem{DC2014} D. Crisan, {The stochastic filtering problem: a brief historical account}, Journal of Applied Probability 51A, 13-22 (2014).

\bibitem{DM} C. Dellacherie and P.-A. Meyer, 
{Probabilities and Potential B: Theory of Martingales}, 
North-Holland Publishing Company (1982).

\bibitem{FH2018} B.P.W. Fernando and E. Hausenblas, 
{Nonlinear filtering with correlated L\'evy noise characterized by copulas}, 
Brazilian Journal of Probability and Statistics 32, 374-421 (2018).

\bibitem{FKK1972} M. Fujisaki,  G. Kallianpur and H. Kunita, 
{Stochastic differential equations for the nonlinear filtering problem}, Osaka Journal Math., 
9-1, 19-40 (1972).

\bibitem{GM2011} B. Grigelionis and R. Mikulevicius, 
{Nonlinear filtering equations for stochastic processes with jumps}, 
in: The Oxford Handbook of Nonlinear Filtering, Oxford University Press (2011).

\bibitem{HL2008}A. J. Heunis and V.  M. Lucic, 
{On the Innovations Conjecture of Nonlinear Filtering with
Dependent Data}, 
Electronic Journal of Probability, Vol. 13, 2190--2216 (2008).

\bibitem{HWY} S. He, J. Wang and J. Yan, {Semimartingale Theory and Stochastic Calculus}, 
Taylor \& Francis, 1992.

\bibitem{IW} Ikeda and Watanabe, {Stochastic Differential Equations and Diffusion Processes}, NorthHolland Publishing Company (1992).

\bibitem{KZ2000} N.V. Krylov and A. Zatezalo, {A direct approach to deriving filtering equations for diffusion processes}, Appl. Math. Optim., 42, no. 3, 315-332 (2000). 

\bibitem{Krylov1979} N.V. Krylov, {On the equivalence of $\sigma$-algebras in the filtering problem of diffusion processes}, Theor. Probab. Appl.,  v.24,  772--781 (1979). 

\bibitem{K} N.V. Krylov, {Introduction to the Theory of Random Processes}, 
American Mathematical Society: Graduate Studies in Mathematics 43 (2002). 

\bibitem{K2002} N.V. Krylov, {A simple proof of a result of A. Novikov}, arXiv:math/0207013v1, 2002. 
 

\bibitem{K2019} N. V. Krylov,  {A few comments on a result of A. Novikov and Girsanov's
theorem}, Stochastics, vol. 91.8, 1186-1189 (2019).  

\bibitem{LS1974} R. S. Liptser and A. N. Shiryayev, {Statistics of Random Processes I.}, 2nd edition, Springer, 2001.

\bibitem{P2005} D.R. Poklukar, {Nonlinear filtering for jump-diffusions}, Journal of Computational and Applied Mathematics 197, 558-567 (2006).

\bibitem{SS2009}S. Popa and S.S. Sritharan, {Nonlinear filtering of It\^o-L\'evy stochastic differential equations with continuous observations}, Communications on Stochastic Analysis 3 (2009).

\bibitem{QD2015} H. Qiao and J. Duan, {Nonlinear filtering of stochastic dynamical systems with {L}\'evy noises}, Advances in Applied Probability 47-3, 902-918 (2015). 

\bibitem{Q2019} H. Qiao, {Nonlinear filtering of stochastic differential equations driven by correlated L\'{e}vy noises}, Stochastics, vol. 93.8,  1156-1185 (2021).

\bibitem{Y} M. Yor, {Sur les th\'eories du filtrage et de la pr\'ediction}, S\'eminaire de probabilit\'es XI, 257-297, Springer (1977).


\end{thebibliography}
\end{document}